\documentclass[12pt]{article}
\usepackage{amssymb}
\usepackage{latexsym}
\usepackage{amsmath}
\usepackage{amsfonts}
\usepackage{amsthm}
\topskip  \makeindex

\newtheorem{theorem}{Theorem}
\newtheorem{proposition}{Proposition}
\newtheorem{corollary}{Corollary}

\newtheorem{remark}{Remark}

\newtheorem{example}{Example}

\title{Whitney algebras and Grassmann's regressive products}

\author{Andrea Brini and Francesco Regonati \\ [18pt]
Dipartimento di Matematica \\ [13pt] ``Alma Mater Studiorum"
Universit\`{a} degli Studi di Bologna}

\begin{document}

\maketitle

\begin{abstract}
Geometric products on tensor powers $\Lambda(V)^{\otimes m}$ of an exterior algebra
and on Whitney algebras \cite{crasch}
provide a rigorous version of Grassmann's {\it regressive products} of 1844 \cite{gra1}.
We study geometric products and their relations with other classical operators on
exterior algebras, such as the Hodge $\ast-$operators and the {\it join} and {\it meet}
products in Cayley-Grassmann algebras \cite{BBR, Stew}.
We establish  encodings of tensor powers $\Lambda(V)^{\otimes m}$ and of Whitney algebras $W^m(M)$
in terms of letterplace algebras and of their geometric products in terms of
divided powers of polarization operators.
We use these encodings to provide simple proofs of the Crapo
and Schmitt exchange relations in Whitney algebras and of two typical classes of identities
in Cayley-Grassmann algebras .
\end{abstract}

\thanks{We thank Henry Crapo and William Schmitt for their advice, encouragement, and invaluable suggestions}

\tableofcontents

\section{Introduction}

In 1986, I. Stewart wrote:
{\it
``The late nineteenth century witnessed many attempts
to develop an algebra of $n-$dimensional space, by analogy with
the representation of the plane by complex numbers.
Prominent among them was Hermann Grassmann's ``Die Lineale Ausdehenungslehre"
(The Calculus of Extension), published in 1844.
Grassmann had the misfortune to write in a discursive, philosophical and obscure style
at the time when axiomatic presentation was becoming de riguer in the mathematical world"
}\cite{Stew}.

As definitively recognized in recent times \cite{Schu, Schu1, pet},
H. G. Grassmann was indeed the inventor of linear and multilinear algebra,
as well as of ``geometry" in {\it arbitrary} finite dimension (\cite{gra1}, 1844).

Grassmann's basic idea was to build up a formal system
that allows geometric entities to be manipulated in an intrinsic (invariant) way,
that is, by making no appeal to a reference system.
In Grassmann's vision, this kind of approach should put together
the synthetic and the analytic approaches to geometry in arbitrary finite dimension.
As Grassmann wrote in 1877:

\begin{quote}
{\it
``Extension theory forms the abstract basis of the theory of space (geometry), that is
it is the pure mathematical science, stripped of all geometric content,
whose special application is that theory.
\begin{footnote}
{
At the time of Grassmann, the terms ``space" and ``geometry"
were used to mean just dimensions one, two and three.
}
\end{footnote}
...
The purpose of this method of calculation in geometry is
to unify the synthetic and analytic methods, that is,
to transplant the advantages of each into the soil of the other, so that
any construction is accompanied by an elementary analytic operation and conversely.
(\cite{gra1}, p. 283 and 285)"
}
\end{quote}

Grassmann's elementary analytic operations were essentially of two kinds:
the {\it outer (progressive) product} (the one that is nowadays called the ``wedge product" in the exterior algebra) and
the {\it regressive product}.
It is worth  claiming that this second kind of product was not a {\it single} operation
but a {\it family} of (unary) operations.

Grassmann himself realized that his general approach
was not only obscure but also misleading for the Mathematicians of his time.
He published in 1862 a second and simplified version of his {\it opus magnum} of 1844.
In this version, he made the deliberate choice of restricting the definition of regressive product
to the special case of those he called {\it real} regressive products (see subsection 3.2, for details).
None the less, the impact of Grassmann's work upon
the mathematical community of the second half of the XIX century
remained almost irrelevant.

There were, to be sure, two bright exceptions:
W. K. Clifford and G. Peano, whose ideas laid the foundations of
the two main algebraic theories that, at present, fulfill
Grassmann's program in the more efficient way.

In his celebrated paper of 1878 \cite{Cliff},
W. K. Clifford introduced what are nowadays called {\it Clifford algebras}.
About one century later D. Hestenes and his school captured (see, e.g \cite{Hest, li})
the various geometric meanings of Clifford algebras in their full generality by introducing
the notion of {\it Geometric Clifford algebra};
these algebras are Clifford algebras in which a distinguished element is chosen, the {\it integral}.
For lack of space, here we don't speak about this important point of view.
We refer the reader to the recent book by H. Li \cite{li}.

In his book of 1888 \cite{Peano}, G. Peano made
another crucial step into an apparently different direction:
he realized that the 1862 regressive product can be defined (at least in dimension 3)
in a transparent and intrinsic way by fixing a {\it bracket}
- a non-degenerate, alternating multilinear  form -
on the ground vector space.
Peano's ideas were at the origin of the modern notion of Cayley-Grassmann algebra (CG-agebras, for short),
developed by G.-C. Rota and his school (see, e.g \cite{BBR}).
In the last decades,
a considerable amount of work has been done dealing with Cayley-Grassmann algebras
(see, e.g. \cite {li}, Chapters 2, 3 and the Bibliography).
These algebras nowadays have important applications both in Mathematics and in Computer Science
(Invariant Theory, Geometric theorem proving  and Computer vision, to name but a few).

In this paper, we provide
a comparative discussion of three classes of algebras
that sprang out from the ideas of Grassmann and Peano, and
some of their main features and typical applications.
These algebras are
the Cayley-Grassmann algebra of a Peano space, endowed with the {\it join} and {\it meet} products;
the tensor powers $\Lambda(V)^{\otimes m}$ of the exterior algebra of a vector space $V,$
endowed with {\it geometric products}, that provide a formalism closer to Grassmann's original approach;
the Whitney algebras $W^m(M)$ of a matroid $M,$ endowed with {\it geometric products}, at first motivated
by the study of representability of matroids.
The study of  tensor powers of exterior algebras and of Whitney algebras   endowed
with geometric products is a recent subject
(\cite{Mou, Moust, crasch, crapo, berget}).

The first main concern of the present paper is the discussion of {\it geometric products} in
the tensor powers $\Lambda(V)^{\otimes m}$ of the exterior algebra of a finite dimensional vector space $V.$
Geometric products provide a rigorous version of Grassmann's regressive products,
both in the {\it real} and in the {\it formal} cases (\cite{gra1}, $\S125$).
It is worth noticing that Grassmann failed in finding a geometric meaning of formal geometric products
\begin{footnote}
{
In 1877, Grassmann wrote:
``In the 1862 Ausdehnungslehre the concept of the formal regressive product is abandoned as sterile,
thus simplifying the whole subject."
\cite{gra1}, p.200
}
\end{footnote};
in Section 5, we exhibit such a geometric meaning,
by showing that geometric products lead to a concise treatment of (generalized) Hodge $\ast-$operators.

The second main concern of our paper is to establish the encoding
of the algebras $\Lambda(V)^{\otimes m}$ and $W^m(M)$ and of their geometric products by means of
quotients of skew-symmetric {\it letterplace algebras} $Skew[\emph{L}|\underline{m}]$ \cite{pri, berget} and
{\it place polarization operators} (Section 7).
Skew-symmetric letterplace algebras are {\it algebras with Straightening law}, and
the action of place polarization operators implements a Lie action of the
{\it general linear Lie algebra} ${\bf gl}(m, \mathbb{Z})$ (Section 6).
This encoding (foreshadowed by G.-C. Rota) allows the treatment of identities to be simplified.

On the one hand, we have that
all the identities among linear combinations of iterated  ``geometric products"
in the algebra $\Lambda(V)^{\otimes m}$ are consequences
of the {\it  Superalgebraic Straightening law} of Grosshans, Rota and Stein \cite{GRS}.
This fact may be regarded as an analog of the Second Fundamental Theorem of Invariant Theory.
The application to Whitney algebras is even more relevant.
The  Straightening law provides a system of $\mathbb{Z}-$linear generators of the ideals that define
the algebras $W^m(M)$ (Subsection 7.3).
Thus, Whitney algebras are still algebras with Straightening law, and we have an algorithm for
the solution of the {\it word problem}
(for example, see the proof of the exchange relations of Crapo and Schmitt, Section 8).
The iterative application of the  Straightening laws gives
rise to an explosive increment of computational complexity; this
is the reason why a systematic investigation of special identities
is called for.

On the other hand, identities that hold for
geometric products may also be derived from identities that hold
in the enveloping algebra ${\bf U}({\bf gl}(m, \mathbb{Z})).$
Thus, the basic ideas borrowed from the method {\it Capelli virtual variables}
(see, e.g. \cite{abl}, \cite{brini}, \cite{bt}, \cite{reg}) can be applied to manipulations of
geometric products (see, e.g. Sections 9 and 10).

\section{The algebras: generalities}

\begin{enumerate}
\item $\mathbf{CG-algebras}$. These algebras are the classical
exterior algebras endowed with the {\it join} product (the
traditional wedge product) and the {\it meet} product. As the join
product of two extensors that represent subspaces $U$ and $V$ in
general position yields an extensor that represents the space
direct sum of $U$ and $V,$ the meet product represents the
intersection space of $U$ and $V.$

In the language of CG-algebras, projective geometry statements and
constructions can be expressed as (invariant) identities and
equations.

\begin{example} (The Desargues Theorem)

In the projective plane ${\mathbf P [{\mathbf K^3}]}$ , consider two triangles $123$ and $1'2'3'$.

Consider the three lines $11'$, $22'$ and $33'$. In the language
of CG-algebras, we have: the  lines     $11'$,  $22'$ and  $33'$
are {\it  concurrent}    if and  only  if
$$
(1 \vee 1')  \wedge (2 \vee 2')    \wedge  ( 3 \vee 3')
$$
equals  ZERO.

Consider the three points $12 \cap 1'2'$, $13  \cap 1'3'$, $23
\cap 2'3'$. These points are {\it collinear} if and only if the
bracket
$$
[(1 \vee 2) \wedge (1' \vee 2'), (1 \vee 3) \wedge (1' \vee 3'),
(2 \vee 3) \wedge (2' \vee 3')]
$$
equals  ZERO.

In the CG-algebra, we have the identity (see, subsection... below)
$$
[(1 \vee 2) \wedge (1' \vee 2'), (1 \vee 3) \wedge (1' \vee 3'),
(2 \vee 3) \wedge (2' \vee 3')] =
$$
$$
- [1, 2, 3][1', 2', 3']  (1 \vee 1')  \wedge (2 \vee 2') \wedge
( 3 \wedge 3').
$$
Thus, the preceding identity implies the following geometric
statement:

The three points
$$
12 \cap 1'2', \ 13  \cap 1'3', \  23 \cap 2'3'
$$
are collinear if and only if the three lines
$$
11',  22',  33'
$$
are concurrent.
\end{example}

It turns out that there is even a third class of  unary
operations, the {\it Hodge star operators}, which is the vector
space analog of complementation in Boolean algebra. The Hodge
$*-$operators are invariantly defined with respect to the {\it
orthogonal group}, and formalize Grassmann's ``Erganzung"; these
operators implement the duality between meet and join.

The notion of CG-algebra endowed with a Hodge $*-$operator is {\it
equivalent} to  that of Geometric algebra in the sense of Hestenes
\cite{Bravi}.

\item
$\mathbf{Tensor\ powers\ of\ the\ exterior\ algebra\ of\ a\ vector\ space }.$
(see \cite{Mou, Moust, crasch, Fau, crapo})
These algebras are endowed with
a product induced by the wedge product, as well as with
a family of linear operators, called {\it geometric products},
that yield a rigorous formulation of Grassmann's 1844 regressive products.
In particular, given two extensors that represent subspaces $U$ and $V,$
{\it not necessarily in general position}, there is a special
geometric product that yields the tensor product of two extensors
that represent the space sum of $U$ and $V$ and the intersection
space of $U$ and $V,$ respectively.

\item
$\mathbf{Whitney\ algebras\ of\ matroids}$
(see \cite{crasch, crapo, berget}).
This class of algebras provides a generalization of
the algebras mentioned in the preceding point, as well as of
the {\it White bracket ring} \cite{whi1}.
Roughly speaking,
Whitney algebras can be regarded as the
generalization of the tensor powers of exterior algebras towards
algebras associated to  systems of points subject to relations of
abstract dependence, in the sense of H.Whitney \cite{whitney}.
Again, in Whitney algebras
the join product and the geometric products are defined.

\end{enumerate}

\section{The algebras: basic constructions}

We provide explicit constructions  of the algebras mentioned
above, and a description of their main features.
First of all, we fix some terminology and notation.

\noindent
Let $V$ be an  $n-$dimensional vector space over a field $\mathbb{K},$
and let $\Lambda(V)$ be its exterior algebra.
An element $A$ of $\Lambda(V)$ that can be written as a product vectors is called an {\it extensor}.
To each representation of a given extensor $A$ as a product of vectors
there corresponds a basis of one and the same dimensional subspace $\overline{A}$ of $V;$
the number of vectors in a representation of $A$ is the {\it step} of $A,$
and equals the dimension of $\overline{A}.$
An extensors $A$ divides an extensor $B,$ in the usual sense,
if and only if the subspace $\overline{A}$ is contained in the subspace $\overline{B}.$
The subspaces of $V$ ordered by set inclusion form a lattice in which, for any two subspaces $U_1, U_2,$
the greatest lower bound $U_1 \frown U_2$ is the subspace set-intersection of $U_1$ and $U_2,$ and
the least upper bound $U_1 \smile U_2$ is the subspace sum of $U_1$ and $U_2.$
This lattice is modular, thus
for any two subspaces $U_1, U_2$ there is a canonical isomorphism between
the intervals $[U_1 \frown U_2, U_1]$ and $[U_2, U_1 \smile U_2].$

\subsection{Cayley-Grassmann algebras}

Let $\mathbb{K}$ be a field. Let $V$ be a vector space over
$\mathbb{K},$ $dim(V) = n,$ endowed with a bracket $[  \ ]$
(a non-degenerate alternating $n-$multilinear form).
The pair $(V, [\ ])$ is called a {\it Peano space}.
The {\it CG-algebra} of $(V, [\ ])$
is the exterior algebra $\Lambda(V)$ endowed with two associative
products:
\begin{enumerate}
\item the usual wedge product, called the {\it join} and denoted
by the symbol $\vee;$ \item the {\it meet}, denoted by the symbol
$\wedge.$ This product is defined in the following way. Let $A$
and $B$ be extensors of steps $a$ and $b,$ respectively, with $a +
b \geq n,$ then:
$$
A \wedge B
=
\sum_{(A)_{(n-b,a+b-n)}} [A_{(1)} B] A_{(2)}
=
\sum_{(B)_{(a+b-n,n-a)}} [A B_{(2)}] B_{(1)}.
$$
If $a + b < n,$ then $A \wedge B = 0$ by definition.

\noindent
(Here we use the Sweedler notation for {\it coproduct slices} in
bialgebras \cite{Swe, Abe}: for an extensor $C$ of step $c,$
$$
\Delta(C)_{(c-h,h)}
=
\sum_{(C)_{(c-h, h)}} C_{(1)}\otimes C_{(2)},
$$
where $step(C_{(1)})= c-h$ and $step(C_{(2)})= h.$)

\noindent
The fact that there are two completely different ways of computing
the meet of two extensors gives a great suppleness and power to
the CG-algebra of a Peano space.
\end{enumerate}

\noindent
The geometric meaning of the join product $\vee$ (the wedge in the
standard language for exterior algebras) is well-known. Let $A$
and $B$ be extensors that represent subspaces $\overline{A}$ and
$\overline{B}$, respectively. If $\overline{A} \frown \overline{B}
= (0)$, then $\overline{A \vee B} = \overline{A} \smile \overline{B};$
otherwise, $A \vee B$ is zero.
Not unexpectedly, the geometric meaning of the meet product
$\wedge$ is the ``dual" of the geometric meaning of the join product.
If $\overline{A} \smile \overline{B} = V$, then $\overline{A \wedge B} = \overline{A}
\frown \overline{B};$ otherwise, $A \wedge B$ is zero.

\begin{example} (The intersection of two lines in $\mathbb{P}^2(\mathbb{K})$)
Let $(V, [\ ])$ be a Peano space over the field $\mathbb{K}$, $dim(V) = 3.$
Let $p_1, p_2, q_1, q_2$ be four mutually non-proportional
vectors in $V$ that span $V.$
The $2-$extensors $p_1 \vee p_2,$ $q_1 \vee q_2$ represent two different lines
in $\mathbb{P}^2(\mathbb{K}).$
The vector
$$
(p_1 \vee p_2) \wedge (q_1 \vee q_2) = [p_1, q_1, q_2]p_2 - [p_2, q_1, q_2]p_1
\\
= - [p_1, p_2, q_1]q_2 + [p_1, p_2, q_2]q_1
$$
represents the intersection point of the two lines.
\end{example}

\subsection{Tensor powers of exterior algebras}

Let $\mathbb{K}$ be a field.
Let $V$ be a vector space over $\mathbb{K},$ $dim(V) = n.$
Let $\Lambda(V) \otimes \Lambda(V)$ be the  tensor square of the exterior algebra $\Lambda(V),$
in the category of $\mathbb{Z}_2-$graded $\mathbb{K}-$algebras.
The product in this algebra is given by
$$
(A_1 \otimes A_2)
(B_1 \otimes B_2)
=
(-1)^{a_2b_1}
A_1B_1 \otimes A_2B_2
$$
for every extensors $A_i, B_j$ of steps $a_i, b_i$ in $\Lambda(V).$
This algebra is endowed with two families of linear operators,
called {\it geometric products}:
\begin{itemize}
\item 'rising' geometric products:
$$
\diamond_{21}^{(h)}:
\Lambda(V) \otimes \Lambda(V) \rightarrow \Lambda(V) \otimes \Lambda(V),
\quad h \in \mathbb{Z}^+
$$
defined by setting, for $A$ and $B$ extensors of steps $a$ and
$b,$ respectively,
$$
\diamond_{21}^{(h)}(A \otimes B)
=
\sum_{(A)_(a-h, h)} A_{(1)}\otimes A_{(2)}B.
$$
\item 'lowering' geometric products:
$$
\diamond_{12}^{(h)}:
\Lambda(V) \otimes \Lambda(V) \rightarrow \Lambda(V) \otimes \Lambda(V),
\quad h \in \mathbb{Z}^+
$$
defined by setting, for $A$ and $B$ extensors of steps $a$ and
$b,$ respectively,
$$
\diamond_{12}^{(h)}(A \otimes B)
=
\sum_{(B)_{(h, b-h)}} AB_{(1)}\otimes B_{(2)}.
$$
\end{itemize}

The Propositions below exploit the geometric meaning of the
geometric products $\diamond_{12}^{(h)},$ $h$ any positive
integer.
In the language of Grassmann, Proposition 1 deals with {\it real
regressive} products (\cite{gra1}, \S 125). The proof of
Proposition 1 is nowadays a simple exercise of multilinear
algebra; nonetheless, it is worth noticing that this Proposition
and its proof are quite close to Grassmann's way of manipulating
extensors (\cite{gra1}, \S 126, \S\S 130-132).
\begin{footnote}
{
For the convenience of the reader, we recall that the terms
{\it magnitude}, {\it system}, {\it nearest covering system} and
{\it common system} of the Ausdehnungslehre corresponds to {\it
extensor}, {\it subspace}, {\it sum} and {\it intersection} of
subspaces.
}
\end{footnote}

\begin{proposition}
Let $A$ and $B$ be extensors, and denote by
$\overline A, \overline B $ the corresponding subspaces of $V.$
In the modular lattice of subspaces of $V,$ consider the isomorphic intervals
$
[\overline{A} \frown \overline{B}, \overline{A}]
\cong
[\overline{B}, \overline{A} \smile \overline{B}],
$
and denote by $p$ their common dimension.

\begin{itemize}
\item For $h = p,$ we have
$$
\diamond^{(h)}_{21} (A \otimes B) = C \otimes D,
$$
where $\overline{C} = \overline{A} \frown \overline{B}$ and
$\overline{D} = \overline{A} \smile \overline{B};$
\item For $h > p,$ we have
$$
\diamond^{(h)}_{21} (A \otimes B) = 0.
$$
\end{itemize}
The operators $\diamond^{(h)}_{12}$ have analogous geometric
meanings.
\end{proposition}

\begin{proof}
Let $C$ and $A'$ be extensors such that $\overline{C} =
\overline{A} \frown \overline{B}$ and $CA' = A,$ $step(A') = p$
(here we use juxtaposition to mean the exterior product of
extensors). Let $a, c$ denote the steps of $A$ and $C,$
respectively. By definition, and since the exterior algebra is a
bialgebra, we have
\begin{multline*}
\diamond^{(h)}_{21} (A \otimes B)
=
\sum_{(a-h,h)} (CA')_{(1)} \otimes (CA')_{(2)}B =
\\
\sum_{h_1+h_2=h}
\left(
\sum_{(C)_{(c-h_1,h_1)}} C_{(1)}\otimes C_{(2)}
\sum_{(A')_{(p-h_2,h_2)}} (A')_{(1)} \otimes (A')_{(2)}
\right)
\left( 1 \otimes B \right)
=
\\
\sum_{(A')_{(p-h,h_)}} C(A')_{(1)} \otimes (A')_{(2)}B.
\qquad (\dag)
\end{multline*}

\begin{itemize}
\item if $h > p,$ the sum $(\dag)$ has no summand, hence reduces
to zero; \item if $h = p,$ formula $(\dag)$ simplifies to $C
\otimes A'B,$ and it turns out that $\overline{A'B} = \overline{A}
\smile \overline{B}.$
\end{itemize}
\end{proof}

\begin{example} (The intersection of two coplanar lines in $\mathbb{P}^n(\mathbb{K})$)
\noindent
Let $V$ be a vector space over the field $\mathbb{K},$ $dim(V) = n + 1\geq 4.$
Let $p_1, p_2, q_1, q_2$ be four mutually non-proportional
vectors in $V$ that span a 3-dimensional subspace of $V.$
The $2-$extensors $p_1p_2,$ $q_1q_2$ represent two different
coplanar lines in $\mathbb{P}^n(\mathbb{K}).$
The problem of finding their intersection point cannot be treated by
the meet in the CG-algebra of a Peano space $(V, [\ ]),$ since
$
(p_1p_2) \wedge (q_1q_2) = 0.
$
However, the problem can be solved in the context of the algebra $\Lambda(V) \otimes \Lambda(V)$
endowed with geometric products.
There are two extensors $u$ and $v$ such that
$$
\diamond^{(1)}_{21} (p_1p_2 \otimes q_1q_2) = p_1 \otimes p_2q_1q_2 - p_2 \otimes p_1q_1q_2 = u \otimes v
$$
where
$$
\overline{u} = \overline{p_1p_2} \frown \overline{q_1q_2} \quad and \quad
\overline{v} = \overline{p_1p_2} \smile \overline{q_1q_2}.
$$
These extensors can be described in function of the vectors $p$'s and $q$'s as follows.
Without loss of generality, we can assume that the extensor
$
p_2q_1q_2
$
represents the plane
$
\overline{p_1p_2} \smile \overline{q_1q_2}.
$
Notice that there exists a scalar $\lambda \in \mathbb{K}$ such that
$
p_1q_1q_2 = \lambda p_2q_1q_2.
$
Now, we have
\begin{multline*}
\diamond^{(1)}_{21} (p_1p_2 \otimes q_1q_2)
=
p_1 \otimes p_2q_1q_2 - p_2 \otimes p_1q_1q_2
=
\\
p_1 \otimes p_2q_1q_2 - \lambda p_2 \otimes p_2q_1q_2
=
(p_1 - \lambda p_2) \otimes p_2q_1q_2.
\end{multline*}
Thus, we can take
$
u = p_1 - \lambda p_2
$
and
$
v = p_2q_1q_2.
$
The vector $p_1 - \lambda p_2$ represents the intersection point of the coplanar lines
$\overline{p_1p_2},$ $\overline{q_1q_2}$ in $\mathbb{P}^n(\mathbb{K}).$
\end{example}

In the language of Grassmann, Proposition 2 will deal with {\it formal
regressive} products (\cite{gra1}, \S 125). Grassmann used the
term {\it formal} since he didn't found any geometric meaning for
them. As a matter of fact, Proposition 2 exhibits such a geometric
meaning, that will turn out to be closely related to the notion of
{\it Hodge $\ast-$operator}.

\begin{proposition}
Let $A$ and $B$ be extensors, and denote by $ \overline A,
\overline B $ the corresponding subspaces of $V.$ In the modular
lattice of subspaces of $V,$ consider the isomorphic intervals
$[\overline{A} \frown \overline{B}, \overline{A}] \cong
[\overline{B}, \overline{A} \smile \overline{B}],$ and denote by $p$
their common dimension.
For $0 \leq h \leq p,$ we have:
\begin{itemize}
\item
the left span of
$
\diamond^{(h)}_{21} (A \otimes B)
$
equals the linear span of all the extensors $H$ that
represent subspaces $\overline{H}$ of codimension $h$ in the
interval $[\overline{A} \frown \overline{B}, \overline{A}];$
\item
the right span of
$
\diamond^{(h)}_{21} (A \otimes B)
$
equals the linear span of all the extensors $K$
that represent subspaces $\overline{K}$ of dimension $h$ in the
interval $[\overline{B}, \overline{A} \smile \overline{B}].$
\end{itemize}

\noindent
The operators $\diamond^{(h)}_{12}$ have analogous geometric
meanings.
\end{proposition}

\begin{proof}
Let $C$ and $A'$ be extensors such that $\overline{C} =
\overline{A} \frown \overline{B}$ and $CA' = A,$ $step(A') = p.$
We have
$$
\diamond^{(h)}_{21} (A \otimes B)
=
\sum_{(A')_{(p-h,h_)}} C(A')_{(1)} \otimes (A')_{(2)}B.
\qquad (\dag)
$$
Assume that
$$
\Delta(A')
=
\sum_{(A')_{(p-h,h)}} (A')_{(1)} \otimes (A')_{(2)}
$$
is a minimal representation of the tensor $\Delta(A');$ thus it
has exactly ${ p \choose h }$ terms, the set of the extensors
$(A')_{(1)}$ is linearly independent, the set of the extensors
$(A')_{(2)}$ is linearly independent.

Then the sum $(\dag)$ has ${ p \choose h }$ terms; each of the
extensors $C(A')_{(1)}$ represents a subspace of codimension $h$
in the interval $[\overline{A} \frown \overline{B},
\overline{A}],$ and the set of these extensors is linearly
independent; each of the extensors $(A')_{(2)}B$ represents a
subspace of dimension $h$ in the interval $[\overline{B},
\overline{A} \smile \overline{B}],$ and the set of these extensors is
linearly independent.

The left span of $\diamond^{(h)}_{21} (A \otimes B)$ is then the
space spanned by the extensors $C(A')_{(1)},$ which turns out to
be the space spanned by all the extensors $H$ that represent
subspaces $\overline{H}$ of codimension $h$ in the interval
$[\overline{A} \frown \overline{B}, \overline{A}].$

The right span of $\diamond^{(h)}_{21} (A \otimes B)$ is then the
space spanned by the extensors $(A')_{(2)}B,$ which turns out to
be the space spanned by all the extensors $K$ that represent
subspaces $\overline{K}$ of dimension $h$ in the interval
$[\overline{B}, \overline{A} \smile \overline{B}].$
\end{proof}

\begin{remark}
We mention that geometric products provide a simple
characterization of the inclusion relation between subspaces.
Given two extensors $A$ and $B,$
$$
\overline{A} \subseteq \overline{B} \quad if \ and \ only \ if
\quad \diamond_{21}^{(1)} (A \otimes B) = 0.
$$
\end{remark}

\begin{example} (Two non-coplanar lines in $\mathbb{P}^n(\mathbb{K}),$ $n \geq 3$)
Let $V$ be a vector space over the field $\mathbb{K},$ $dim(V) n + 1 \geq 4.$
Let $p_1, p_2, q_1, q_2$ be four vectors in $V$
that span a 4-dimensional subspace of $V.$
The $2-$extensors $p_1p_2,$ $q_1q_2$ represent two
non-coplanar lines  in $\mathbb{P}^n(\mathbb{K}).$
The tensor
$$
\diamond^{(1)}_{21}  (p_1p_2 \otimes q_1q_2) = p_1 \otimes p_2q_1q_2 - p_2 \otimes p_1q_1q_2
$$
is not a decomposable tensor.
This representation is a minimal one.

\noindent
The left span of $\diamond^{(1)}_{21} (p_1p_2 \otimes q_1q_2)$ is
the vector space spanned by the set of vectors $\{p_1, p_2 \},$ that equals
the vector space spanned by the vectors that represent
the points in the projective line $\overline{p_1p_2}.$

\noindent
The right span of $\diamond^{(1)}_{21} (p_1p_2 \otimes q_1q_2)$ is
the vector space spanned by the set of $3-$extensors $\{p_2q_1q_2, -p_1q_1q_2 \},$ that equals
the vector space spanned by the $3-$extensors that represent
the planes that contain the projective line $\overline{q_1q_2},$
and are contained in the projective subspace $\overline{p_1p_2q_1q_2 }$.
\end{example}

More generally, let $\Lambda(V)^{\otimes m}$ be the  $m-$th tensor power of the exterior algebra $\Lambda(V),$
in the category of $\mathbb{Z}_2-$graded $\mathbb{K}-$algebras.
On this algebra we have a family of geometric products
$$
\diamond_{ij}^{(h)}:
\Lambda(V)^{\otimes m} \rightarrow \Lambda(V)^{\otimes m},
$$
for every $m = 2, 3, \ldots$ and $1\leq i, j \leq m,$ $i \neq j,$
defined as follows.

\noindent
Let $i <j;$ for any $A_1, \ldots, A_m$ extensors of steps $a_1,
\ldots, a_m.$ The rising geometric products are given by
\begin{multline*}
\diamond^{(h)}_{ji}
\left(
A_1 \otimes \cdots \otimes A_i \otimes \cdots \otimes A_j \otimes \cdots \otimes A_m
\right)
=
\\
=
(-1)^{h ( a_{i+1} + \cdots + a_{j-1})}
\sum_{\left(A_i\right)_{(a_i-h, h)}}
A_1 \otimes \cdots \otimes
\left(A_i\right)_{(1)} \otimes \cdots \otimes \left(A_i\right)_{(2)} A_j
\otimes \cdots \otimes A_m.
\end{multline*}

\noindent
The lowering geometric products are given by
\begin{multline*}
\diamond^{(h)}_{ij}
\left(
A_1 \otimes \cdots \otimes A_i \otimes \cdots \otimes A_j \otimes \cdots \otimes A_m
\right)
=
\\
= (-1)^{h ( a_{i+1} + \cdots + a_{j-1})}
\sum_{\left(A_i\right)_{(h, a_j - h)}}
A_1 \otimes \cdots \otimes
A_i \left(A_j\right)_{(1)} \otimes \cdots \otimes \left(A_j\right)_{(2)}
\otimes \cdots \otimes A_m.
\end{multline*}

\subsection{Whitney algebras}

Let $M = M(S)$ be a matroid of rank $n$ over a set $S.$ Let
$Skew(S)$ be the free skew-symmetric algebra over $\mathbb{Z}$ on the set $S,$ and let
$Skew(S)^{\otimes m}$ be the $m-$th tensor power algebra of the $\mathbb{Z}-$algebra $Skew(S),$
in the category of $\mathbb{Z}_2-$graded algebras.

\noindent
For each  {\it dependent} subset $\{v_1, \ldots, v_p\}$ in $M = M(S),$ we
consider the monomial $w = v_1 \cdots v_p$ in $Skew(S)$ and all
its possible {\it slices} in $Skew(S)^{\otimes m}:$
$$
\Delta_{(i_1, \ldots, i_p)}(w)
=
\sum_{(w)_(i_1, \ldots, i_p)}
w_{(1)}\otimes \cdots \otimes w_{(m)},
$$
where $i_1 + \ldots + i_m = p.$
Let $I(M)$ be the (bilateral) ideal of $Skew(S)^{\otimes m}$ generated by all these slices.

\noindent
The $m-$th {\it Whitney algebra} of $M$ is the quotient algebra
$$
W^m(M) = Skew(S)^{\otimes m}/I(M).
$$
Notice that in $W^1(M)$ words corresponding to independent subsets of $S$
that span the same flat are not necessarily equal up to a scalar.
For every tensor $u_1 \otimes u_2 \otimes \cdots \otimes u_m$ in $Skew(S)^{\otimes m},$
its image in $W^m(M)$ is denoted by $u_1 \circ u_2 \circ \cdots \circ u_m.$

\noindent
In the tensor power algebra $Skew(V)^{\otimes m},$
{\it geometric products} are defined
in strict analogy with those defined in $\Lambda(V)^{\otimes m},$ and are
still denoted by $\diamond_{ij}^{(h)},$ for every
$m = 2, 3, \ldots$ and $1\leq i, j \leq m,$ $i \neq j.$

Indeed, {\it geometric products} $\diamond_{ij}^{(h)}$ are well-defined in the Whitney algebra $W^m(M).$
In the original approach of Crapo and Schmitt \cite{crasch},
this is a non-trivial fact founded on the theory of {\it lax Hopf algebras}.
In Section 7, we show that this fact directly follows from
the letterplace encoding of Whitney algebras
in combination with
the definition of geometric products in terms of divided powers of {\it place polarization operators}.

\section{Relations among the algebras}

\subsection{From tensor powers of exterior algebras to CG-Algebras}

Let $V$ be a vector space, $dim(V) = n.$ In the exterior algebra
$\Lambda(V)$ we {\it fix} an extensor $E$ of step $n.$ $E$ is
usually called the {\it integral}. The choice of the integral
induces the choice of a bracket on $V,$ by setting:
$$
[x_1, \ldots, x_n]E = x_1 \cdots x_n, \quad for \ every \ x_1,
\ldots, x_n \in V.
$$
(Here we use juxtaposition to denote the wedge product).

Thus, the vector space $V$ becomes a Peano space $(V, [ \ ]).$ The
meet product in the CG-algebra of $(V, [ \ ])$ is recovered from
the geometric product in $\Lambda(V)^{\otimes 2}$ by means of the
following identity. Let $A$ and $B$ be extensors of steps $a$ and
$b$, respectively. We have
\begin{multline*}
\diamond_{21}^{(n-b)} (A \otimes B)
=
\sum_{(A)_{(a+b-n, n-b)}}
A_{(1)}\otimes A_{(2)}B
=
\\
\sum_{(A)_{(a+b-n, n-b)}}
A_{(1)}[A_{(2)}B]\otimes E
=
(-1)^{(a+b-n)(n-b)} (A \wedge B) \otimes E
%
\end{multline*}
Clearly,
$
\diamond_{21}^{(a)} (A \otimes B) = 1 \otimes (A \vee B).
$

\noindent
Similarly, we have
\begin{multline*}
\diamond_{12}^{(n-a)} (A \otimes B)
=
\sum_{(B)_{(n-a, a+b-n)}} AB_{(1)}\otimes B_{(2)}
\\
=
\sum_{(B)_{(n-a, a+b-n)}} E \otimes [AB_{(1)}] B_{(2)}
=
(-1)^{(a+b-n)(n-a)} E \otimes (A \wedge B)
%
\end{multline*}
Clearly,
$
\diamond_{12}^{(b)} (A \otimes B) = (A \vee B) \otimes 1.
$

\subsection{ From Whitney algebras to tensor powers of exterior algebras.}

Let $M = M(S)$ be a matroid on a finite set $S,$ and $V$ a vector
space over some field $\mathbb{K}.$ A {\it representation} of $M$ in $V$
is a mapping $ g : S \rightarrow V $ such that a set $A \subseteq
S$ is independent in $M$ iff $g$ is one-to-one on $A$ and the set
$g(A)$ is linearly independent in $V.$ A matroid that admits a
representation is said to be a {\it representable matroid}. In
this case, for every $m \in \mathbb{Z}^+,$ there is exactly one
ring morphism $\hat{g}^m : W^m(M) \rightarrow \Lambda^{\otimes
m}(V)$ such that
$$
\hat{g}^m
(1\circ \cdots \circ x \circ 1 \circ \cdots \circ 1)
=
1\otimes \cdots \otimes g(x)\otimes 1 \otimes \cdots \otimes 1
$$
for all $x$ in $S.$
Furthermore we have:
$$
\hat{g}^{m} \circ \diamond^{(h)}_{ij}
=
\diamond^{(h)}_{ij} \circ \hat{g}^{m}.
$$
for every $h \in \mathbb{Z}^+,$ and every $1\leq i, j \leq m,$
with $i \neq j.$

A matroid $M$ is representable if and only if
in each Whitney algebra $W^m(M), m = 1, 2, \ldots$ of $M$
each product
$
w_1 \circ w_2 \circ \ldots \circ w_m
$
of words $w_i$ corresponding to independent subsets of $S$
is not zero
(\cite{crasch} Corollary 8.9, p. 256).


\section{The Hodge operators}

\subsection{The Hodge $\ast-$operators}


Let $V$ be a vector space, and
let $(f_1, \ldots, f_n)$ be an ordered basis of $V.$
To each subset
$I = \{i_1,  \ldots , i_k \}$
of the set $\{1,  \ldots , n\},$
where $1 \leq i_1 < \cdots < i_k \leq n,$
corresponds a {\it canonical extensor} $f_I = f_{i_1} \vee \cdots \vee f_{i_k};$
for $I = \emptyset$ we have $f_I = 1,$ and
for $I = \{1,  \ldots , n\}$ we set $f_I = F.$
When there is no danger of misunderstanding,
we will identify a subset with the corresponding canonical extensor,
and instead of $f_I$  we will write $I.$

\noindent
The {\it Hodge star operator} associated to the
basis $(f_1, \ldots, f_n)$ is the linear operator
$\ast : \Lambda(V) \rightarrow \Lambda(V)$ whose values on the canonical
extensors are given by
$$
\ast A = \epsilon_A A',
$$
where $A'$ is the complementary subset of $A$ in $\{1,  \ldots , n\},$ and
$\epsilon_A$ is the sign implicitly defined by $F = \epsilon_A A A'.$
We have $\ast 1 = F$ and $\ast F = 1;$
notice that
$$
\ast (\ast A) = (-1)^{k(n-k)} A, \qquad k = step(A).
$$

\begin{remark}
The Hodge $\ast-$operators are defined by choosing a linearly
ordered basis of the vector space $V$. A natural question arises:
under which conditions two Hodge $\ast-$operators coincide? The
answer is a classical result.

\noindent
Let $\emph{F} = (f_1, \ldots, f_n)$ and
$\emph{F}' = (f'_1, \ldots, f'_n)$ be two ordered bases of $V,$ and
let $\ast$ and $\ast'$ be the Hodge operators associated to them, respectively.
The operators $\ast$ and $\ast'$ coincide if and only if the
transition matrix from  $\emph{F}$ to  $\emph{F}'$ is an
orthogonal matrix (see. e.g, \cite{BBR}, Theorem 6.2).
\end{remark}


We recall that the Hodge $\ast-$operators implement
the duality between join and meet
(\cite{BBR}, Theorem 6.3).
More specifically, in the exterior algebra $\Lambda(V)$ we {\it fix}
an {\it integral} $E,$ and, consequently, a bracket $[\ ].$
Thus, the vector space $V$ becomes a Peano space $(V, [ \ ]),$ and
we can consistently consider its CG-algebra.
Then we have the identities
\begin{align*}
[F] \ast (A \vee B) &=  (\ast A) \wedge (\ast B).
\\
[F]^{-1} \ast (A \wedge B) &=  (\ast A) \vee (\ast B).
\end{align*}
Now, assume that $\emph{F} = (f_1, \ldots, f_n )$ is a {\it unimodular} basis of $V,$
that is $[f_1, \ldots, f_n] = 1,$ or, equivalently,
$F = f_1 \vee \cdots \vee f_n = E.$
Then the preceding identities specialize to
\begin{align*}
\ast (A \vee B) &=  (\ast A) \wedge (\ast B).
\\
\ast (A \wedge B) &=  (\ast A) \vee (\ast B).
\end{align*}


The preceding results generalize to a result on geometric products.

\begin{theorem}
Let $\ast$ be the Hodge star operator associated to a basis of the vector space $V.$
We have the identity
$$
(\ast \otimes \ast) \diamond_{21}^{(h)} \left( A \otimes B \right)
= (-1)^{h(a+ b-n)} \diamond_{12}^{(h)} (\ast \otimes \ast) \left(
A \otimes B \right),
$$
where $A$ and $B$ are extensors of steps $a$ and $b,$ respectively.
\end{theorem}

\begin{proof}
On the one hand, we have
\begin{multline*}
(\ast \otimes \ast) \diamond_{21}^{(h)} \left( A \otimes B \right)
= (\ast \otimes \ast) \sum_H \zeta'_H \zeta''_H (A \cap H')
\otimes (B \cup H),
\\
= \sum_H \zeta'_H \zeta''_H \epsilon_{A \cap H'} \epsilon_{B
\cup H} (A' \cup H) \otimes (B' \cap H')
\end{multline*}
where $H$ runs through the $h-$subsets of $A \cap B'$ and
$$
A = \zeta'_H (A \cap H') H,
\qquad
H B = \zeta''_H (B \cup H).
$$
On the other hand, we have
\begin{multline*}
\diamond_{12}^{(h)} (\ast \otimes \ast) \left( A \otimes B \right)
= \diamond_{12}^{(h)} \left( \epsilon_A \epsilon_{B} A' \otimes
B' \right)
\\
= \sum_H \epsilon_A \epsilon_{B} \eta'_H \eta''_H (A' \cup H)
\otimes (B' \cap H'),
\end{multline*}
where $H$ runs through the $h-$subsets of $B' \cap A$ and
$$
B' = \eta'_H H (B' \cap H')
\qquad
A' H = \eta''_H (A' \cup H).
$$
A rather tedious sign computation gives
$$
\zeta'_H \zeta''_H \epsilon_{B\cup H} \epsilon_{A\cap H'} =
(-1)^{h(a+ b-n)} \epsilon_A \epsilon_B \eta'_H \eta''_H.
$$
\end{proof}

In the last part of this subsection we will describe the way to
specialize  the preceding Theorem in order to obtain the main
result that relates Hodge operators to join and meet in
CG-algebras.

Let $V$ be a vector space, $dim(V) = n.$
In the exterior algebra $\Lambda(V)$ we {\it fix} an {\it integral} $E,$ and,
consequently, a bracket $[ \ ].$ Thus, the vector space $V$
becomes a Peano space $(V, [ \ ]),$ and we can consistently
consider its CG-algebra.

\noindent
In the previous Theorem, by setting $h = a,$ the step of $A,$
we have that the left hand side becomes
$$
(\ast \otimes \ast) \diamond_{21}^{(a)}
\left(A \otimes B \right)
=
(\ast \otimes \ast)
\left(1 \otimes (A \vee B) \right)
=
F \otimes \ast (A \vee B)
=
[F]E \otimes \ast (A \vee B),
$$
the right hand side becomes
$$
\diamond_{12}^{(a)} (\ast \otimes \ast)
\left(A \otimes B \right)
=
\diamond_{12}^{(a)}
\left((\ast A) \otimes (\ast B) \right)
=
(-1)^{a(a+ b-n)}  E \otimes ((\ast A) \wedge (\ast B)),
$$
thus we get
$$
[F] \ast (A \vee B) =  (\ast A) \wedge (\ast B).
$$
In the previous Theorem, by setting $h = n -b,$ the costep of $B,$
we have that the left hand side becomes
\begin{multline*}
(\ast \otimes \ast) \diamond_{21}^{(n - b)}
\left( A \otimes B \right)
=
(-1)^{(a+ b-n)(n - b)}
(\ast \otimes \ast)
\left((A \wedge B) \otimes E \right)
\\
=
(-1)^{(a+ b-n)(n - b)} [F]^{-1}
\left(\ast (A \wedge B) \otimes 1 \right),
\end{multline*}
the right hand side becomes
$$
\diamond_{12}^{(n - b)} (\ast \otimes \ast)
\left(A \otimes B \right)
=
\diamond_{12}^{(n - b)}
\left( (\ast A) \otimes (\ast B) \right)
=
\left( (\ast A) \vee (\ast B) \right) \otimes 1,
$$
thus we get
$$
[F]^{-1}   \ast (A \wedge B)  =  (\ast A) \vee (\ast B).
$$

\begin{remark}
From Remark 1, it follows that the Hodge $\ast-$operator induces
an antitone mapping on the lattice of subspaces.
Indeed, by setting in the previous Theorem $h=1,$ we have
$$
(\ast \otimes \ast) \diamond_{21}^{(1)} \left( A \otimes B \right)
= (-1)^{(a+ b-n)} \diamond_{12}^{(1)} (\ast \otimes \ast) \left( A
\otimes B \right).
$$
The left hand side equals zero if and only if $\overline{A} \subseteq
\overline{B};$ the right hand side equals zero if and only if
$\overline{\ast A} \supseteq \overline{\ast B}.$
Thus,
$$
\overline{A} \subseteq \overline{B} \quad if \ and \ only \ if \quad
\overline{\ast A} \supseteq \overline{\ast B}.
$$
\end{remark}

\subsection{The generalized Hodge operators}


To each tensor $t$ in the tensor product of two vector spaces,
there are associated two spaces $L_t,$ $R_t,$ its left and right spans,
a class of isomorphism $L_t \rightarrow R_t,$ and
a canonical nonsingular bilinear map $L_t \times R_t \rightarrow \mathbb{K}.$
We will exploit these constructions and invariants
for the geometric product of two extensors in $\Lambda(V),$
which is a tensor in $\Lambda(V) \otimes \Lambda(V).$

Let $(f_1, \ldots,  f_n)$ be a linearly ordered basis of a vector space $V,$
set $F = f_1 \vee \cdots \vee f_n,$ and let $[\ ]$ be the corresponding bracket.
Consider the tensor
$$
\left(\sum_{h=0}^n \diamond_{21}^{(h)}\right) (F\otimes 1)
=
\sum_{(F)} F_{(1)}\otimes F_{(2)},
$$
and notice that:
the left span and the right span are $\Lambda(V);$
for any minimal representation in which $F_{(1)}$ and $F_{(2)}$ are subwords of $F,$
the associated linear map $\Lambda(V) \rightarrow \Lambda(V)$ is
the Hodge $\ast-$operator associated to the basis $(f_1, \ldots,  f_n);$
the canonical bilinear mapping $\Lambda(V) \times \Lambda(V) \rightarrow \mathbb{K}$ is the {\it cap-product,}
defined by by setting $X \times Y \mapsto [XY],$ for any $X, Y$
extensors of complementary steps, and $X \times Y \mapsto 0$
otherwise.


In the following, we will see how these facts generalize to any geometric product.
Consider two extensors $A, B.$ Let $A = CA'$ be a factorization of
$A$ as the product of two extensors, where $C$ is an extensor such
that $\overline{C} = \overline{A} \frown \overline{B},$ and set
$D = A'B,$ so that $\overline{D} = \overline{A} \smile \overline{B}.$
Denote by $p$ the step of $A'.$

Consider the tensor
$$
\left(\sum_{h=0}^n \diamond_{21}^{(h)}\right) (A\otimes B)
=
\sum_{(A')} CA'_{(1)}\otimes A'_{(2)}B
$$
where
$$
\Delta (A') = \sum_{(A')} A'_{(1)}\otimes A'_{(2)}.
$$
Choose a minimal representation of $\Delta(A'),$
by representing the extensor $A'$ as an exterior product
$
a_1 \vee \cdots \vee a_p,
$
$a_i \in V,$ and writing it as
$$
\Delta(A') = \sum_{i=0}^r \epsilon_i A'_{i,(1)}\otimes A'_{i,(2)},
$$
where $A'_{i,(1)}$ and $A'_{i,(2)}$ are {\it increasing subwords}
of the word $a_1 \vee \cdots \vee a_p$ (in the preceding coproduct
expression, we explicitly write the sign of the summands; it is a
slight modification of the Sweedler notation that will turn to be
useful in the subsequent part of this paragraph. The symbol
$\epsilon_i$ represents a {\it sign}, whose meaning should be
obvious).

\noindent
The representation
$$
\left(\sum_{h=0}^n \diamond_{21}^{(h)}\right) (A\otimes B) =
\sum_{i = 1}^r \epsilon_i CA'_{i,(1)}\otimes A'_{i,(2)}B \quad
(\dag)
$$
is independent, thus minimal. Thus we have:

\begin{enumerate}

\item
The left span
is
$
L
=
\left\langle
CA'_{i,(1)};\ i = 1, \ldots, r
\right\rangle
$
which is the linear span of all the extensors
that represent subspaces in the interval
$[\overline{A} \frown \overline{B}, \overline{A}].$

\noindent
The right span
is
$
R
=
\left\langle
\epsilon_i  A'_{i,(2)}B;\ i = 1, \ldots, r
\right\rangle,
$
which is the linear span of all the extensors
that represent subspaces in the interval
$[\overline{B}, \overline{A} \smile \overline{B}].$

\item
The {\it generalized Hodge $\ast_{\dag}-$operator} associated to
the minimal representation $(\dag)$ is defined to be the linear
operator
$$
\ast_{\dag} : L \rightarrow R,
\qquad
\ast_{\dag} :  CA'_{i,(1)} \mapsto \epsilon_i A'_{i,(2)}B,
\quad i = 1,\ldots, r.
$$
Consider the factorization
$$
\ast_{\dag} = \varphi \circ \widetilde{\ast}.
$$
where
\begin{align*}
\widetilde{\ast} &: L \rightarrow L,
\qquad
\widetilde{\ast} :  CA'_{i,(1)} \mapsto \epsilon_i CA'_{i,(2)},
\quad i = 1,\ldots, r,
\\
\varphi &: L \rightarrow R,
\qquad
\varphi :  CH' \mapsto H'B.
\end{align*}
Notice that: the linear map $\widetilde{\ast}$ is antitone, since
it can be identified with the Hodge $\ast-$operator
$$
\ast : \Lambda\left(\overline{A'}\right) \rightarrow \Lambda\left(\overline{A'}\right),
\qquad
\ast : A'_{i,(1)} \mapsto \epsilon_i  A'_{i,(2)};
$$
the linear map $\varphi$ is monotone.
Thus the generalized Hodge $\ast_{\dag}-$operator
$\ast_{\dag} = \varphi \circ \widetilde{\ast}$
is antitone.

\item
The canonical pairing
$\beta : L \times R \rightarrow \mathbb{K}$
is defined by
$$
\beta \left( CA'_{i,(1)},\ \epsilon_j A'_{j,(2)}B \right)
=
\delta_{ij}, \qquad i,j= 1, \ldots, r.
$$
We have also the following intrinsic characterization of $\beta.$

\noindent
Let $X$ be an extensor in $L,$
let $Y$ be an extensor in $R,$ and
denote by $k$ the relative step of $X$ with respect to $C.$
Then
$$
\diamond^{(k)}_{21} (X \otimes Y) = \beta(X, Y)  C \otimes D.
$$
Indeed, for $X = CA'_{i,(1)}$ and $Y = \epsilon_j A'_{j,(2)}B,$
we have
$$
\diamond^{(k)}_{21}
(CA'_{i,(1)} \otimes \epsilon_j A'_{j,(2)}B)
=
\epsilon_j
C \otimes A'_{i,(1)}A'_{j,(2)}B
=
\delta_{ij}  C \otimes D.
$$
\end{enumerate}

\section{Letterplace algebras and polarization operators}

The algebras $\Lambda(V)^{\otimes m}$ and $W^m(M)$ and their
geometric products admit natural and efficient descriptions in
terms of (skew-symmetric) {\it Letterplace algebras} and {\it
place polarization operators} (see, e.g. \cite{brini}).

\subsection{Skew-symmetric Letterplace algebras and place polarization operators.}

First of all, take
a set of symbols $\emph{L}$ and
the set $\underline{m} = \{1, 2, \ldots, m\}$ of the first $m$ positive integers;
the  set $[\emph{L}|\underline{m}] = \{ (x|i); x \in \emph{L}, \ i \in \underline{m} \}$ is called
the set of {\it letterplace variables}.
The elements of $\emph{L}$ are called {\it letters},
the elements of $\underline{m}$ are called {\it places}.

The skew-symmetric Letterplace algebra
$
Skew[\emph{L}|\underline{m}]
$
is the free associative skew-symmetric (unitary) $\mathbb{Z}-$algebra
generated by the set $[\emph{L}|\underline{m}].$

Given $h, k \in \underline{m},$
the {\it place polarization} operator from $h$ to $k$ is
the unique {\it derivation} $\emph{D}_{kh}$ of $Skew[\emph{L}|\underline{m}]$ such that
$$
\emph{D}_{kh} (x|i)  = \delta_{hi} (x|k),
$$
for every  $x \in \emph{L},$ $i \in \underline{m}.$
In particular, the action of a polarization on a monomial is given by
\begin{multline*}
\emph{D}_{kh} \left( (x_{i_1}|j_1) (x_{i_2}|j_2) \cdots
(x_{i_n}|j_n) \right) =
\\
= \sum \delta_{h j_r} (x_{i_1}|j_1) \cdots (x_{i_{r-1}}|j_{r-1})
(x_{i_r}|k) (x_{i_{r+1}}|j_{r+1}) \cdots (x_{i_n}|j_n).
\end{multline*}

\noindent
The place polarization operators satisfy the commutation relations
$$
\emph{D}_{kh}\emph{D}_{ji} -
\emph{D}_{ji}\emph{D}_{kh}
=
\delta_{hj} \emph{D}_{ki} - \delta_{ki} \emph{D}_{jh},
$$
thus they implement a representation of
the general linear Lie algebra $gl(m, \mathbb{Z})$ on
$Skew[\emph{L}|\underline{m}].$

\subsection{Biproducts in skew-symmetric letterplace algebras and
divided powers of polarization operators}

We provide a brief description of the divided power notation for
the {\it biproducts}
$$
(w|i_1^{(q_1)} i_2^{(q_2)} \cdots  i_p^{(q_p)}),
\qquad
w\ a\ word\ on\ \emph{L},
\quad
i_1, \ldots, i_p \in \underline{m},
$$
in the skew-symmetric letterplace algebra
$Skew[\emph{L}|\underline{m}].$
In the following, we will denote the length of a word $w \in Mon[\emph{L}]$ by the symbol $|w|.$
\begin{itemize}
\item
Let $p, q \in \mathbb{Z}^+,$
$x_1, x_2, \ldots, x_p \in \emph{L},$
$i \in \underline{m}.$
Then
$$
(x_1 x_2 \cdots x_p|i^{(q)})
=
(x_1|i)(x_2|i) \cdots (x_p|i),
$$
if $p = q,$ and is zero otherwise.
We also use the convention that
$$
(w|i^{(0)})
=
\mathbf{1},
\quad
\mathbf{1}\ the\ unity\ of\ the\ algebra,
$$
if $w$ is the empty word, and equals zero otherwise.

\item
(Laplace expansion)
Let $w$ be a word on the alphabet $\emph{L}.$
Let $q_1, \ldots, q_m$ be nonnegative integers.
Then
$$
(w|i_1^{(q_1)} i_2^{(q_2)}\cdots i_p^{(q_p)})
=
\sum_{(w)_{(q_1,\ldots, q_p)}} (w_{(1)}|i_1^{(q_1)}) (w_{(2)}|i_2^{(q_2)})\cdots (w_{(p)}|i_p^{(q_p)}),
$$
if $|w| = q_1 + \cdots + q_p,$ and is  zero otherwise.
\end{itemize}

\noindent
The biproducts are skew-symmetric in the letters:
$$
(\cdots x_i\cdots x_j\cdots|i_1^{(q_1)}\cdots i_p^{(q_p)})
=
-
(\cdots x_j\cdots x_i\cdots|i_1^{(q_1)}\cdots i_p^{(q_p)}),
$$
and symmetric in the places:
$$
(x_1 \cdots  x_t|\cdots i_s^{(q_s)} \cdots i_t^{(q_t)} \cdots)
=
(x_1 \cdots  x_t|\cdots i_t^{(q_t)} \cdots i_s^{(q_s)} \cdots);
$$
thus, in particular, each biproduct can be written in the form
$$
(w|1^{(q_1)} 2^{(q_2)} \cdots m^{(q_m)}).
$$

\begin{remark}
\noindent
The biproducts satisfy the anticommutation relations
$$
(w|i_1^{(p_1)} \cdots  i_s^{(p_s)})
(w'|j_1^{(q_1)} \cdots j_t^{(q_t)})
=
(-1)^{pq}
(w'|j_1^{(q_1)} \cdots  j_t^{(q_t)})
(w|i_1^{(p_1)} \cdots  i_s^{(p_s)}),
$$
where $p = \sum p_i$ and $q = \sum q_i.$
\end{remark}

Now we consider the action of place polarization operators on biproducts.

\begin{proposition}
For each nonnegative integer $h$ and$1\leq i <j \leq m,$ we have
$$
\emph{D}^{h}_{ji} (w|\cdots i^{(q_i)} \cdots j^{(q_j)} \cdots)
=
\left\{
{
{
\frac {(q_{j} + h)!} {q_j!} (w|\cdots i^{(q_i - h)}
\cdots j^{(q_j + h)} \cdots)
}
\atop
\
0
}
\right.
$$
according to $h \leq q_i,$ or $h > q_i.$
An analogous result holds for $i>j.$
\end{proposition}

\noindent
Let $i, j \in {\underline m},$ with $i \neq j.$
Given a nonnegative integer $h,$ the {\it $h-$th divided power}
$
\emph{D}^{(h)}_{ji}
$
of the place polarization
$
\emph{D}^{(h)}_{ji}
$
is, by definition, the operator
$$
\emph{D}^{(h)}_{ji} = \frac {\emph{D}^{(h)}_{ji}} {h!}.
$$
We claim that the next corollary implies that
the action of a divided power of a place polarization operator
is well-defined on the algebra $Skew[\emph{L}|\underline{m}].$

\begin{corollary}
For each natural integer $h$ and$1\leq i <j \leq m,$ we have
$$
\emph{D}^{(h)}_{ji} (w|\cdots i^{(q_i)} \cdots j^{(q_j)} \cdots)
=
\left\{
{
{
{{q_{j} + h} \choose q_j} (w|\cdots i^{(q_i - h)} \cdots j^{(q_j + h)} \cdots)
}
\atop
\
0
}
\right.
$$
according to $h \leq q_i,$ or $h > q_i.$ In particular,
$$
\emph{D}^{(h)}_{ji} (w|k^{(q)})
= \delta_{ik} (w|i^{(q-h)}j^{(h)}),
$$
An analogous result holds for $i>j.$
\end{corollary}

\subsection{The Straightening Law}

In the context of the skew-symmetric Letterplace algebra $Skew[\emph{L}|\underline{m}],$
the supersymmetric Straightening Law of Grosshans, Rota and Stein \cite{GRS}
can be stated as follows.

\begin{theorem}
Let $u, v, w$ be words on $\emph{L};$ let $p_1, \ldots, p_m$ and
$q_1, \ldots, q_m,$ be $m-$tuples of nonnegative integers, and set
$\sum p_i = p$ and $\sum q_i = q.$ Then
\begin{multline*}
\sum_{(v)}
(u v_{(1)}|1^{(p_1)} \cdots m^{(p_m)})
(v_{(2)} w|1^{(q_1)} \cdots m^{(q_m)}) =
(-1)^{|u||v|} \times
\\
\sum_{(u), r}
(-1)^{|u_{(2)}|} c_r
(v u_{(1)}|1^{(p_1 + r_1)}\cdots m^{(p_m + r_m)})
(u_{(2)} w|1^{(q_1 - r_1)} \cdots m^{(q_m - r_m)}),
\end{multline*}
where
$c_r = \prod_i {{p_i + r_i} \choose {r_i}}.$
The sum in the left-hand side is extended to all the slices of the
word $v;$
the sum in the right-hand side is extended to all the slices of
the word $u,$ and to all the $m-$tuples $r_1, \ldots, r_m$ such
that $0\leq r_i \leq q_i.$
\end{theorem}

\begin{corollary}
Let $u, v$ be words on $\emph{L};$ let $1 \leq i \neq j \leq m,$
and $p, q$ nonnegative integers. Then
\begin{multline*}
\sum_{(v)}
(u v_{(1)}|i^{(p)})
(v_{(2)}|j^{(q)})
=
(-1)^{|u||v|}
\sum_{(u), r}
(-1)^{|u_{(2)}|}
(v u_{(1)}|i^{(p)} j^{(r)})
(u_{(2)}|j^{(q - r)}).
\end{multline*}
The sum in the left-hand side is extended to all the slices of the
word $v;$
the sum in the right-hand side is extended to all the slices of
the word $u,$ and to all the nonnegative integers $r$ such that
$0\leq r \leq q.$
\end{corollary}

\section{Letterplace encodings}

\subsection{The letterplace encoding of the algebra $Skew[\emph{L}]^{\otimes m}$}

Given a set $\emph{L},$ consider
the free associative skew-symmetric (unitary) $\mathbb{Z}-$algebra $Skew[\emph{L}]$
generated by the set $\emph{L}.$
On each $\mathbb{Z}-$algebra tensor power $Skew[\emph{L}]^{\otimes m},$ $m \in \mathbb{Z}^+,$
we may consider the ``formal" geometric products $\diamond^{(h)}_{ji},$
defined and denoted in the same way as in Subsection 3.2.

\noindent
There is a ${\mathbb{Z}}$-algebra  isomorphism
\begin{align*}
\Phi: Skew[\emph{L}|\underline{m}] &\rightarrow Skew[\emph{L}]^{\otimes m},
\\
(x|i) &\mapsto 1 \otimes \cdots \otimes 1 \otimes x \otimes 1
\otimes \cdots \otimes 1,
\end{align*}
(where $x \in \emph{L}$ is in the $i-$th fold in the tensor product).

\begin{proposition}
The geometric products on $Skew[\emph{L}]^{\otimes m}$
correspond, via the isomorphism $\Phi,$ to
the divided powers of place-polarizations on $Skew[\emph{L}|\underline{m}]:$
$$
\diamond_{ji}^{(h)}
=
\Phi \circ \emph{D}_{ji}^{(h)} \circ \Phi^{-1},
\qquad
i \neq j,
\quad
h \in \mathbb{Z}^+.
$$
\end{proposition}

\begin{proof}
We consider the case of raising geometric products, i.e. the case $i < j.$
We will prove that the operators
$\diamond_{ji}^{(h)} \circ \Phi$ and $\Phi \circ \emph{D}_{ji}^{(h)}$
have the same action on the monomials of the type
$$
(A_1|1^{(q_1)}) \cdots (A_i|i^{(q_i)}) \cdots (A_j|j^{(q_j)}) \cdots (A_m|m^{(q_m)}).
$$
On the one hand, we have
\begin{multline*}
\diamond^{(h)}_{ji}
\left(
\Phi
\left(
\cdots (A_i|i^{(q_i)}) \cdots (A_j|j^{(q_j)}) \cdots
\right)
\right)
=
\\
\diamond^{(h)}_{ji}
\left(
\cdots \otimes A_i \otimes \cdots \otimes A_j \otimes \cdots
\right)
=
\\
=
(-1)^{h ( q_{i+1} + \cdots + q_{j-1})}
\sum_{(A_i)_{(q_i-h, h)}} \cdots \otimes
(A_i)_{(1)} \otimes \cdots \otimes (A_i)_{(2)} A_j
\otimes \cdots.
\end{multline*}
On the other hand, we have
\begin{multline*}
\Phi
\left(
\emph{D}_{ji}^{(h)}
\left(
\cdots
(A_i|i^{(q_i)}) \cdots (A_j|j^{(q_j)}) \cdots
\right)
\right)
=
\\
\Phi
\left(
\cdots
(A_i|i^{(q_i-h)}j^{(h)}) \cdots (A_j|j^{(q_j)})
\cdots
\right)
=
\\
\Phi \left( \sum_{(A_i)_{(q_i-h, h)}} \cdots
((A_i)_{(1)}|i^{(q_i-h)}) ((A_i)_{(2)}|j^{(h)}) \cdots (A_j|j^{(q_j)})
\cdots \right)
=
\\
(-1)^{h ( q_{i+1} + \cdots + q_{j-1})}
\Phi \left( \sum_{(A_i)} \cdots
((A_i)_{(1)}|i^{(q_i-h)}) \cdots ((A_i)_{(2)}A_j|j^{(q_j+h)})
\cdots  \right)=
\\
= (-1)^{h ( q_{i+1} + \cdots + q_{j-1})}
\sum_{(A_i)} \cdots \otimes
(A_i)_{(1)} \otimes \cdots \otimes (A_i)_{(2)} A_j
\otimes \cdots.
\end{multline*}
\end{proof}

\noindent
In particular, we have that
each geometric product $\diamond_{ji}^{(1)}$ correspond to a polarization $\emph{D}_{ij},$ for $i\neq j;$
ths leads us to define
the $\diamond_{ji}^{(1)}$ as the correspondent of the polarization $\emph{D}_{ii}.$

We write each geometric product $\diamond^{(1)}_{ji}$ briefly as $\diamond_{ji},$ and
call it a {\it simple} geometric product, for $1\leq i, j \leq m.$
The encoding of tensor powers of free skew-symmetric algebras and geometric products with
skew-symmetric letterplace algebras and polarization operators leads
to the following characterization and properties of simple geometric products.

\begin{proposition}
The simple geometric products on
the algebra $Skew[\emph{L}]^{\otimes m}$
enjoy the following properties:
\begin{itemize}
\item
Each $\diamond_{ji}$ is
a derivation on the algebra $Skew[\emph{L}]^{\otimes m};$
it is the unique derivation such that
\begin{multline*}
\diamond_{ji}
( \cdots \otimes 1 \otimes a \otimes 1 \cdots \otimes 1 \otimes 1 \otimes 1 \otimes \cdots )
=
\\
\delta_{ih}
\cdots \otimes 1 \otimes 1 \otimes 1 \cdots \otimes 1 \otimes a \otimes 1 \otimes \cdots ,
\end{multline*}
for any $a \in \emph{L}.$
In the left hand side $a$ occurs in the $h-$th fold,
in the right hand side $a$ occurs in the $j-$th fold.

\item
The space spanned by the $\diamond_{ij}$ for $1 \leq i,j \leq m$ is
a Lie subalgebra of the Lie algebra of derivations on $Skew[\emph{L}]^{\otimes m};$
furthermore, simple geometric products satisfy the commutation relations
$$
\diamond_{ij} \diamond_{hk} - \diamond_{hk} \diamond_{ij} =
\delta_{jh} \diamond_{ik} - \delta_{ki} \diamond_{hj}.
$$
Thus, the mapping that sends
each $m\times m$ elementary matrix with $1$ in the $(i,j)-$th entry to
the simple geometric product $\diamond_{ij}$
induces a Lie representation of
the general linear Lie algebra $gl(m, \mathbb{Z})$ in
the Lie algebra of derivations of
$Skew[\emph{L}]^{\otimes m}.$
\end{itemize}
\end{proposition}

\noindent
Notice that the geometric products are divided powers of simple geometric products:
$$
\diamond^{(h)}_{ij}
=
\frac {\diamond^h_{ij}} {h!},
\qquad
i \neq j.
$$

\subsection{The Letterplace encoding of $\Lambda(V)^{\otimes m}.$}

Let $\mathbb{K}$ be a field, and $V$ a finite-dimensional $\mathbb{K}-$vector space.
For each tensor power
$
\Lambda(V)^{\otimes m}
$
of the exterior algebra of $V,$
we have the isomorhisms
$$
\Lambda(V)^{\otimes m}
\cong
\frac {\mathbb{K}\otimes Skew[V]^{\otimes m}}   {I_m(V)}
\cong
\frac {\mathbb{K}\otimes Skew[V|\underline{m}]}   {J_m(V)}
$$
where:
$I_m(V)$ is the bilateral ideal of the algebra
$
\mathbb{K}\otimes Skew[V]^{\otimes m}
$
generated by the elements of the form
$
1\otimes \cdots \otimes v \otimes \cdots \otimes 1
-
\lambda \cdot 1\otimes \cdots \otimes v_1 \otimes \cdots \otimes 1
-
 \mu \cdot 1\otimes \cdots \otimes v_2 \otimes \cdots \otimes 1
$
for every $v = \lambda v_1 + \mu v_2 \in V,$ and every fold of the tensor;
$J_m(V)$ is the bilateral ideal of the algebra
$
\mathbb{K}\otimes Skew[V|\underline{m}]
$
generated by the elements of the form
$
(v|i)
-
\lambda
(v_1|i)
-
\mu
(v_2|i)
$
for every $v = \lambda v_1 + \mu v_2 \in V,$ and every $1\leq i \leq m.$
Thus we have an isomorphism
$$
\Phi_V:
\frac {\mathbb{K}\otimes Skew[V|\underline{m}]}   {J_m(V)}
\rightarrow
\Lambda(V)^{\otimes m}.
$$
Now,
each place polarization on the algebra
$
\mathbb{K}\otimes Skew[V|\underline{m}]
$
leaves invariant the ideal $J_m(V),$
thus it induces a place polarization on the quotient algebra
$
(\mathbb{K}\otimes Skew[V|\underline{m}])/J_m(V).
$


\noindent
From the results of   subsection 7.1 we get, in particular:
\begin{itemize}
\item
The geometric products on $\Lambda(V)^{\otimes m}$
correspond, via the isomorphism $\Phi_V,$ to
the divided power place polarizations on
$
(\mathbb{K}\otimes Skew[V|\underline{m}])/J_m(V):
$
$$
\diamond_{ji}^{(h)} = \Phi_V \circ \emph{D}_{ji}^{(h)} \circ \left(\Phi_V\right)^{-1};
$$
\item
the mapping that sends
each $m\times m$ elementary matrix with $1$ in the $(i,j)-$th entry to
the simple geometric product $\diamond_{ij}$
induces a Lie representation of
the general linear Lie algebra $gl(m, \mathbb{K})$ in
the Lie algebra of derivations of
$\Lambda(V)^{\otimes m}.$
\end{itemize}

\subsection{The Letterplace encoding of $W^m(M).$}

Let $M = M(S)$ be a matroid of rank $n$ over a set $S.$
Consider the $m-$th Whitney algebra of $M,$
$$
W^m(M) = Skew(S)^{\otimes m}/I(M),
$$
where $I(M)$ is the bilateral ideal of $Skew(S)^{\otimes m}$ generated by
the slices of dependent words on $S.$

\noindent
Under the isomorphism
$$
\Phi:
Skew[S|\underline{m}] \rightarrow Skew(S)^{\otimes m},
$$
the ideal $I(M)$ in $Skew(S)^{\otimes m}$ corresponds to
the ideal $J(M)$ in $Skew[S|\underline{m}]$ generated by the biproducts
$
(w|1^{(p_1)} \cdots m^{(p_m)}),
$
for every word $w = x_1 x_2\cdots x_p$ whose
corresponding set $\{x_1, x_2, \ldots, x_p\}\subseteq S$ is dependent in $M.$
Indeed,
\begin{align*}
\Phi (w|1^{(p_1)} \cdots m^{(p_m)})
&=
\Phi
\sum_{(w)_{(p_1, \ldots, p_m)}} (w_{(1)}|1^{(p_1)}) \cdots (w_{(m)}|m^{(p_m)})
\\
&=
\sum_{(w)_{(p_1, \ldots, p_m)}} w_{(1)}\otimes \cdots \otimes w_{(p)}.
\end{align*}
Thus we have an isomorphism
$$
\Phi_M: Skew[S|\underline{m}]/J(M) \rightarrow Skew(S)^{\otimes m}/I(M) = W^m(M).
$$
We refer to the algebra
$$
\mathcal{W}^m(M) = Skew[S|\underline{m}]/J(M)
$$
as the {\it letterplace encoding} of $W^m(M).$

Now, under the isomorphism
$$
\Phi:
Skew[S|\underline{m}] \rightarrow Skew(S)^{\otimes m},
$$
the geometric products on $Skew(S)^{\otimes m}$ correspond to
the divided power polarizations on $Skew[S|\underline{m}].$
The divided power polarizations on $Skew[S|\underline{m}]$ leave invariant the ideal $J(M),$
so they are well-defined on the quotient algebra $\mathcal{W}^m(M).$
Thus,
the geometric products on $Skew(S)^{\otimes m}$ leave invariant the ideal $I(M),$
so they are well-defined on the quotient algebra,
the Whitney algebra $W^m(M).$

\noindent
From the results of subsection 7.1 we get, in particular:
\begin{itemize}
\item
the mapping that sends
each $m\times m$ elementary matrix with $1$ in the $(i,j)-$th entry to
the simple geometric product $\diamond_{ij}$
induces a Lie representation of
the general linear Lie algebra $gl(m, \mathbb{Z})$ in
the Lie algebra of derivations of
$W^m(M).$
\end{itemize}

\section{The Exchange relations}

This subsection is devoted to a discussion of what Crapo and Schmitt
called {\it the fundamental exchange relations in the Whitney algebra}
(\cite{crasch}, p. 218 and Theorem 7.4). These relations, in particular,
generalize the two definitions of the meet in the CG-algebra of a Peano space.


The original approach to Whitney algebras makes use of
the heavy categorical machinery of ``lax Hopf algebras".
The proof of the exchange relations is based  on the ``Zipper Lemma" (\cite{crasch}, Theorem 5.7, p. 240),
which is an identity satisfied by the homogeneous components of coproducts
in a free skew-symmetric algebra.



Thanks to the letterplace encoding of the Whitney algebra,
the exchange relations turn out to be an almost direct consequence of the
Straightening Laws of Grosshans, Rota and Stein
\cite{GRS}.


\begin{proposition}
Let $M(S)$ be a matroid on the set $S.$
Let $u$ and $v$ be words of legth $p$ and $q$ on $S,$
associated to independent subsets $A$ and $B,$ respectively.
Let $k$ be the nonnegative integer such that
$
\rho(A \cup B) + k = p + q.
$
In the algebra $\mathcal{W}^2(M)$ we have
\begin{multline*}
\sum_{(v)_{q-kk,k}}
(uv_{(1)}|1^{(p+q-k)}) (v_{(2)}|2^{(k)})
=
\\
(-1)^{pq + k}
\sum_{(u)_{p-k,k}}
(vu_{(1)}|1^{(p+q-k)}) (u_{(2)}|2^{(k)}).
\end{multline*}
\end{proposition}


\begin{proof}
From Corollary 2 (Subsect. 6.3), we have
\begin{multline*}
\sum_{(v)_{q-kk,k}}
(uv_{(1)}|1^{(p+q-k)}) (v_{(2)}|2^{(k)})
=
\\
(-1)^{pq}
\sum_{t=0}^k
(-1)^{t}
\sum_{(u)_{p-t,t}}
(vu_{(1)}|1^{(p+q-k)}2^{(k-t)}) (u_{(2)}|2^{(t)}).
\end{multline*}
Now, for each $t,$
to each coproduct slice $(u)_{p-t,t}$
corresponds a word $vu_{(1)}$ of length $q+p-t;$
from the assumption
$
\rho(A \cup B) = p + q - k
$
it follows that for $t < k$
the word $vu_{(1)}$ corresponds to a dependent subset of $S,$
or contains a repeated symbol;
in any case the biproduct
$
(vu_{(1)}|1^{(p+q-k)}2^{(k-t)})
$
vanishes for $t < k.$
Thus the right hand side of the above equality simplifies to
$$
(-1)^{pq+k}
\sum_{(u)_{p-k,k}}
(vu_{(1)}|1^{(p+q-k)}) (u_{(2)}|2^{(k)}).
$$
\end{proof}

By translating the previous Proposition in the Whitney algebra $W^2(M),$
and making a trivial sign computation, we get

\begin{theorem}
(The exchange relations of Crapo and Schmitt \cite{crasch}, Thm.7.4)

Under the same assumptions of the preceding Proposition,
in the algebra $W^2(M)$ we have:
$$
\sum_{(u)_{q-k, k}}
uv_{(2)}\circ v_{(1)}
=
\sum_{(v)_{p-k, k}}
u_{(1)}v \circ u_{(2)}.
$$
\end{theorem}

\section{Alternative Laws}

One of the more traditional and significant themes (on the path traced by Grassmann) in the study
Grassmann-Clifford Geometric Calculus consists in the study of invariant identities
that describe geometrical statements and constructions
(see e.g. \cite{gra1} Appendix III p. 285ff. (1877), \cite{Ford}, \cite{Hawr}, \cite{Main}, \cite{li}).
In order to make effective this approach, one has to develop a systematic work of the the identities
that hold for meet and join of extensors in a CG-algebra; these identities were called in \cite{BBR}
the {\it alternative laws}.

In the following,
for any extensors $A, B$ of steps $a, b,$ we set
$$
A \dot{\wedge} B
:=
\sum_{(A)_{(a+b-n, n-b)}}
A_{(1)}[A_{(2)}B]
=
(-1)^{(a+b-n)(n-b)}
A \wedge B.
$$
The operation $\dot{\wedge}$ is associative up to a sign.
From now on,
each expression involving $\dot{\wedge}$ is thought
as nested from left to right.
Notice that
\begin{align*}
\diamond_{21}^{(n-b)} (A \otimes B)
&=
(A \dot{\wedge} B) \otimes E;
\\
\diamond_{12}^{(b)} (A \otimes B)
&=
(A \vee B) \otimes 1.
\end{align*}

\subsection{A permanental identity}

From the Lie commutation relations, one gets that the
simple geometric products with indexes $0, i_1, \ldots, i_r, j'_1,
\ldots, j'_r,$ where $0$ is distinct from each of the other
indexes and each of the $i's$ is distinct from each of the $j's,$
satisfy the identity
$$
\diamond_{j'_r 0} \cdots \diamond_{j'_1 0}
\diamond_{0 i_r} \cdots \diamond_{0 i_1}
=
\mathrm{per} \left(\diamond_{j'_s i_t}\right)
+ \sum_{s=1}^r \diamond_{\ast \ast} \cdots  \diamond_{j'_s 0}.  \qquad (\ddag)
$$
The operator in formula $(\ddag)$ is said to be a {\it permanental Capelli operator}.
The left-hand side is called  'virtual form' of the right hand
side; the right-hand side is called an expansion of the left-hand
side, and the terms of the sum are the ``Capelli queues" of the
expansion (see, e.g. \cite{abl}, \cite{reg}).

\noindent
The identities we will discuss in the next two subsections
turn out to be consequences of different Laplace expansions
of the permanental term
$
\mathrm{per} \left(\diamond_{j'_s i_t}\right)
$
in formula $(\ddag).$

\subsection{Two typical alternative laws in CG-algebras}

One of the simplest alternative laws is the following (see, e.g. \cite{BBR}, Corollary 7.2):

\begin{proposition}
In the Cayley-Grassmann algebra of a Peano space $(V, [\ ])$ we have:
$$
\left(a_1 \vee \cdots \vee a_r\right) \wedge B_{1'} \cdots \wedge B_{r'}
=
\sum_{\sigma \in S_r} (-1)^{|\sigma|}
\left(a_{\sigma(1)} \wedge B_{1'}\right) \cdots
\left(a_{\sigma(r)} \wedge B_{r'}\right)
$$
for any vectors $a_1, \ldots, a_r$ in $V = \Lambda^1(V)$ and any
covectors $B_{1'}, \ldots, B_{r'}$ in $\Lambda^{n-1}(V)
\cong V^*.$
\end{proposition}

\begin{proof}
By specializing the identity $(\ddag)$
on the tensor power of exterior algebra $\Lambda(V)$ with
folds indexed $0, 1, \ldots, r,$ $ 1', \ldots, r',$
we have the operator identity
$$
\diamond_{r'0} \cdots \diamond_{1'0}
\diamond_{0r} \cdots \diamond_{01}
=
\sum_{\sigma \in S_r} \diamond_{r'\sigma_r} \cdots  \diamond_{1'\sigma_1}
+
\sum_{s=1}^r  \diamond_{\ast \ast} \cdots \diamond_{s'0}.
$$
By evaluating the expression on the left hand side
on the tensor
$
1 \otimes a_1 \otimes \cdots \otimes a_r \otimes B_{1'} \otimes \cdots \otimes B_{r'},
$
we get
\begin{align*}
\phantom{=} &
\diamond_{r'0} \cdots \diamond_{1'0}
\diamond_{0r} \cdots \diamond_{01}
\left(
1 \otimes a_1 \otimes \cdots \otimes a_r
\otimes B_{1'} \otimes \cdots \otimes B_{r'}
\right)
\\
=&
\diamond_{r'0} \cdots \diamond_{1'0}
\left(
(a_1 \vee \cdots \vee a_r)
\otimes 1 \otimes \cdots \otimes 1 \otimes B_{1'} \otimes \cdots \otimes B_{r'}
\right)
\\
=&
(-1)^{{r \choose 2}(n+1)}
\left((a_1 \vee \cdots \vee a_r) \wedge B_{1'} \wedge
\cdots \wedge B_{r'} \right) \otimes 1 \otimes \cdots \otimes
1 \otimes E \otimes \cdots \otimes E.
\end{align*}
By evaluating the expression on the right hand side
on the tensor
$
1 \otimes a_1 \otimes \cdots \otimes a_r \otimes B_{1'} \otimes \cdots \otimes B_{r'},
$
since the Capelli queues vanish on this tensor,
we get
\begin{align*}
\phantom{=} &
\left[
\sum_{\sigma \in S_r}
\diamond_{r'\sigma_r}\cdots \diamond_{1'\sigma_1}
\right]
\left(
1 \otimes a_1 \otimes \cdots \otimes a_r
\otimes B_{1'} \otimes \cdots \otimes B_{r'}
\right)
\\
=&
\sum_{\sigma \in S_r}
(-1)^{{r \choose 2}(1+n) + |\sigma|}
\left(a_{\sigma_1} \wedge B_{1'}\right) \cdots
\left(a_{\sigma_r} \wedge B_{r'}\right) \otimes 1 \otimes \cdots
\otimes 1 \otimes E \otimes \cdots \otimes E.
\end{align*}
The statement follows by equating the last steps of the two evaluations,
and clearing up the signs.
\end{proof}

One of the deepest alternative laws is the last one
in (\cite{BBR}, Corollary 7.11); it can be viewed as playing a role in the CG-algebra
similar to the role of the distributive law for union and intersection of sets in Boolean algebra.

\begin{proposition}
In the Cayley-Grassmann algebra of a Peano space $(V, [\ ]),$
let $C_1, \ldots, C_r$ be extensors of steps $n-q_1 > 0, \ldots, n-q_r > 0,$ $q_i >0.$ Set $p = q_1 + \cdots + q_r.$
Let $A$ and $B$ be extensors, $step(A) = s,$ $step(B) = k,$ $s + k = p.$
Then
$$
(A \vee B) \wedge (C_1 \wedge \cdots \wedge C_r)
=
A \wedge
\sum_{(i_1, \ldots, i_r)} \varepsilon_{i_1, \ldots, i_r}
\sum_{(B)}
(B_{(1)} \vee C_1) \wedge \cdots \wedge (B_{(r)} \vee C_r),
$$
where: the first sum is taken over all the $r-$tuples $(i_1,
\ldots, i_r) \vdash s,$ the second sum is taken over all the
$r-$slices of $B$ of type $(q_1-i_1, \ldots, q_r-i_r) \vdash p-s =
k,$ and
$$
\varepsilon_{i_1, \ldots, i_r}
=
sg \left( \sum_{r\geq h > k \geq
1} i_h(q_k-i_k) \right).
$$
\end{proposition}

\begin{proof}
Let us consider the tensor $ 1 \otimes A \otimes B
\otimes C_1 \otimes \cdots \otimes C_r $ in the tensor space of
$r+3$ copies of the exterior algebra $\Lambda(V),$ where the copies
are indexed $0, 1, 2, 3, \ldots, 2+r.$

\noindent
We have
\begin{align*}
& \phantom{=} \diamond_{2+r,0}^{(q_r)}  \cdots  \diamond_{30}^{(q_1)}  \diamond_{02}^{(k)}  \diamond_{01}^{(s)}
 \left( 1 \otimes A \otimes B \otimes C_1
\otimes \cdots \otimes C_r \right)
\\
=& \diamond_{2+r,0}^{(q_r)}  \cdots \diamond_{30}^{(q_1)}  \left( AB \otimes 1 \otimes 1 \otimes C_1 \otimes
\cdots \otimes C_r \right)
\\
=& \beta_1 \times \left( AB \dot{\wedge} C_1  \cdots
\dot{\wedge} \cdots \dot{\wedge} C_r \right) \times 1 \otimes 1
\otimes 1 \otimes E \otimes \cdots \otimes E
\\
=& \gamma_1 \beta_1 \times \left( AB \wedge C_1 \wedge C_2 \wedge
\cdots \wedge C_r \right) \times 1 \otimes 1 \otimes 1 \otimes E
\otimes E \otimes \cdots \otimes E,
\end{align*}
where
$$
\beta_1 = sg \left( n \sum_{r\geq h \geq 1} q_h (h-1) \right),
\qquad \gamma_1 = sg \left( \sum_{r\geq h > k \geq 1} q_h q_k
\right).
$$
Notice that the operator $\diamond_{2+r,0}^{(q_r)}  \cdots  \diamond_{30}^{(q_1)}  \diamond_{02}^{(k)}
\diamond_{01}^{(s)}$ is a  ``normalized" permanental Capelli operator in the sense of formula $(\ddag).$
Thus, by applying a Laplace expansion (of steps $(s, k)$), we get:
\begin{multline*}
\diamond_{2+r,0}^{(q_r)}  \cdots  \diamond_{30}^{(q_1)}  \diamond_{02}^{(k)}  \diamond_{01}^{(s)}
=
\diamond_{2+r,0}^{(q_r)}  \cdots  \diamond_{30}^{(q_1)}  \diamond_{01}^{(s)} \diamond_{02}^{(k)}
\\
= \sum_{(i_1, \ldots, i_r)} \left(\diamond_{2+r, 1}^{(i_r)} \cdots \diamond_{31}^{(i_1)} \right)
\times
\left(\diamond_{2+r, 2}^{(q_r-i_r)} \cdots \diamond_{32}^{(q_1-i_1)}\right) +
\sum  \diamond_{\ast, \ast} \cdots \diamond_{\ast 0},
\end{multline*}
where the first sum is taken over all the $r-$tuples $(i_1, \ldots, i_r) \vdash s,$
and the Capelli queues $ \diamond_{\ast, \ast} \cdots \diamond_{\ast 0}$
annihilate the tensor $1 \otimes A \otimes B \otimes C_1 \otimes C_2
\otimes \cdots \otimes C_r.$
Thus, we have
\begin{align*}
& \phantom{=}
\diamond_{2+r,0}^{(q_r)} \cdots \diamond_{30}^{(q_1)}
\diamond_{02}^{(k)} \diamond_{01}^{(s)}
\left(
1 \otimes A \otimes B \otimes C_1 \otimes \cdots \otimes C_r
\right)
\\
&=
\sum_{(i_1, \ldots, i_r)}
\diamond_{2+r, 1}^{(i_r)} \cdots \diamond_{31}^{(i_1)}
\times
\diamond_{2+r, 2}^{(q_r-i_r)} \cdots \diamond_{32}^{(q_1-i_1)}
\left(
1 \otimes A \otimes B \otimes C_1 \otimes \cdots \otimes C_r
\right)
\\
&=
\sum_{(i_1, \ldots, i_r)}
\alpha
\sum_{(B)}
\diamond_{2+r, 1}^{(i_r)} \cdots \diamond_{31}^{(i_1)}
\left(
1 \otimes A \otimes 1 \otimes B_{(1)}C_1 \otimes \cdots \otimes B_{(r)}C_r
\right)
\\
&=
\sum_{(i_1, \ldots, i_r)}
\alpha \beta_2
\sum_{(B)}
\left(
A \dot{\wedge} B_{(1)}C_1 \dot{\wedge} \cdots \dot{\wedge} B_{(r)}C_r
\right)
\times
1 \otimes 1 \otimes 1 \otimes E \otimes \cdots \otimes E
\\
&=
\sum_{(i_1, \ldots, i_r)}
\alpha \beta_2 \gamma_2
\sum_{(B)}
\left(
A \wedge B_{(1)}C_1 \wedge \cdots \wedge B_{(r)}C_r
\right)
\times
1 \otimes 1 \otimes 1 \otimes E \otimes \cdots \otimes E
\end{align*}
where the first sum is taken over al the $r-$tuples $(i_1, \ldots,
i_r) \vdash s,$ the second sum is taken over all the $r-$slices of
$B$ of type $(q_1-i_1, \ldots, q_r-i_r) \vdash p-s = k,$ and
$$
\alpha = sg \left( \sum_{r \geq h > k \geq 1} (q_h - i_h) (n -
q_k) \right), \quad \beta_2 = sg \left( n\sum_{r \geq h \geq 1}
i_h (h-1) \right),
$$
$$
\gamma_2 = sg \left( \sum_{r \geq h > k
\geq 1} i_h i_k \right).
$$
Thus we have
$$
\gamma_1 \beta_1 \times \left( AB \wedge C_1 \wedge C_2 \wedge
\cdots \wedge C_r \right) = \sum \alpha \beta_2 \gamma_2 \sum (
A \wedge B_{(1)} C_1 \wedge \cdots \wedge B_{(r)} C_r ),
$$
and it turns out that
$$
\beta_1 \gamma_1 \alpha \beta_2 \gamma_2 = \varepsilon.
$$
\end{proof}

\section{The modular law}

In the exterior algebra $\Lambda(V)$ of a vector space $V,$ let
$A, B, C$ be extensors, where $A$ divides $C.$ Thus $\overline{A},
\overline{B}, \overline{C}$ are subspaces of $V,$ with
$\overline{A} \subseteq \overline{C},$ and in the lattice of subspaces
of $V$ we have the modular law
$$
(\overline{A} \smile \overline{B}) \frown \overline{C} =
\overline{A} \smile (\overline{B} \frown \overline{C}).
$$
It is natural to expect that this law has a counterpart in terms
of geometric products on the tensor power $\Lambda(V)^{\otimes
3}.$

Consider the pair of integers $p, q$ with
the following geometric meanings:
\begin{align*}
p &=& \rho \left(\overline{A}/(\overline{A} \frown \overline{B})\right)
&=& \rho \left((\overline{A} \smile \overline{B} \frown
\overline{C})/(\overline{B} \frown \overline{C})\right) &=& \rho
\left((\overline{A} \smile \overline{B})/\overline{B}\right)
\\
q &=& \rho \left((\overline{B} \smile \overline{C})/\overline{C}\right)
&=& \rho \left((\overline{A} \smile \overline{B})/(\overline{A}
\smile \overline{B} \frown \overline{C})\right) &=& \rho
\left(\overline{B}/(\overline{B} \frown \overline{C})\right).
\end{align*}

On the one hand, we have
$$
\diamond_{32}^{(q)} \diamond_{21}^{(p)} \left( A \otimes B \otimes
C \right) = A_1 \otimes B_1 \otimes C_1,
$$
where
\begin{align*}
\overline{A_1} &= \overline{A} \frown \overline{B}
\\
\overline{B_1} &= \left(\overline{A} \smile \overline{B}\right)
\frown \overline{C}
\\
\overline{C_1} &= \overline{A} \smile \overline{B} \smile
\overline{C}.
\end{align*}

On the other hand, we have
$$
\diamond_{21}^{(p)} \diamond_{32}^{(q)} \left( A \otimes B \otimes
C \right) = A_2 \otimes B_2 \otimes C_2,
$$
where
\begin{align*}
\overline{A_2} &= \overline{A} \frown \overline{B} \frown
\overline{C}
\\
\overline{B_2} &= \overline{A} \smile \left(\overline{B} \frown
\overline{C}\right)
\\
\overline{C_2} &= \overline{B} \smile \overline{C}.
\end{align*}

Thus, $\overline{A_1} = \overline{A_2},$ $\overline{B_1} =
\overline{B_2},$ $\overline{C_1} = \overline{C_2},$ and it is
natural to expect that $ A_1 \otimes B_1 \otimes C_1 = A_2 \otimes
B_2 \otimes C_2, $ i.e., that
$$
\diamond_{32}^{(q)} \diamond_{21}^{(p)} \left( A \otimes B \otimes
C \right) = \diamond_{21}^{(p)} \diamond_{32}^{(q)} \left( A
\otimes B \otimes C \right)
$$
This identity is true, and turns out to be a special case of a more general one.
Given any pair of non-negative integers $s, t$, if  $A$ divides $C,$ we have
\begin{multline*}
\frac {\diamond_{32}^t} {t!}   \frac {\diamond_{21}^s} {s!} (A \otimes B \otimes C) =
\left(\frac {\diamond_{32}^{t-1}} {t!}  \diamond_{21}  \diamond_{32}
\frac {\diamond_{21}^{s-1}} {s!} +
\frac {\diamond_{32}^{t-1}} {t!}  \frac {\diamond_{21}^{s-1}} {s!}  \diamond_{31}\right)
(A \otimes B \otimes C) =
\\
= \frac {\diamond_{32}^{t-1}} {t!}  \diamond_{21}  \diamond_{32}
\frac {\diamond_{21}^{s-1}} {s!}(A \otimes B \otimes C) = \ldots =
\frac {\diamond_{21}^s} {s!} \frac {\diamond_{32}^t} {t!}(A \otimes B \otimes C),
\end{multline*}
since, under the assumption that $A$ divides $C,$ we have
$$
\diamond_{31} \left( A \otimes B \otimes C \right) = 0.
$$

\section{Appendix: Left span, right span, and other invariants of a tensor}

In this section, we recall some facts about
basic invariants of a tensor
in the tensor product of two spaces.

\begin{enumerate}

\item
Let $V$ and $W$ be two vector spaces, and
let $V \otimes W$ their tensor product.
Any tensor $t \in V \otimes W$ has various
representations. The left span of a representation
$$
t = \sum_{i=1}^p v_i \otimes w_i
$$
of $t$ is the space $\langle v_1, \ldots, v_p \rangle$ generated
by the vectors of $V$ occurring in the representation.
The {\it left span} of $t$ is the space $L_t$
intersection of the left spans of all the representations of $t.$

\noindent
If the vectors $w_1, \ldots, w_p$ of $W$
occurring in the representation are independent,
the representation is called right-independent.
Any tensor admits a right-independent representation.

\begin{proposition}
The left span of a right-independent representation of a tensor
$t$ is contained in the left span of any representation of $t.$ As
a consequence, the left span $L_t$ of $t$ is the left span of any
right-independent representations of $t.$
\end{proposition}

\item
Analogous concepts and results are obtained by exchanging left and right.

\noindent
We use the term ``independent representation" instead of
``(left and right)-independent representation".
Any tensor admits an independent representation.

\begin{proposition}
The number of summands of an independent representation of a tensor $t$ is
less than or equal to the number of summands in any other representation of $t;$
when equality holds, the latter representation is independent.
\end{proposition}

\noindent
Due to the preceding Proposition, an independent representation
is also said to be a {\it minimal representation}.

\item
In any left-independent representation of a tensor $t,$
$$
t = \sum_{i=1}^p v_i \otimes w_i
$$
the list $(w_1, \ldots, w_p)$ of the vectors occurring on the right
is uniquely determined by
the list $(v_1, \ldots, v_p)$ of the vectors occurring on the left.
The linear mapping
$$
L_t \rightarrow R_t,
\qquad
v_i \mapsto w_i,
\quad
i = 1,\ldots, p
$$
is an isomorphism.

\begin{proposition}
Let $\phi: L_t \rightarrow R_t$ and $\phi': L_t \rightarrow R_t$ be
the linear isomorphism associated to two independent representations
$$
t = \sum_{i=1}^p v_i \otimes w_i = \sum_{i=1}^p v'_i \otimes w'_i
$$
of a tensor $t.$ Then $\phi = \phi'$ if and only if
the transition matrix for the basis
$(v_1, \ldots, v_p)$ to the basis $(v'_1, \ldots, v'_p)$ of $L_t$
is an orthogonal matrix.
\end{proposition}

\item
From the previous item,
there is no canonical linear mapping from the left span to the right span of a tensor.
There is, instead, a canonical bilinear mapping on
the product of the left span by the right span.

\begin{proposition}
Let $t$ be a tensor in $V\otimes W,$ and let
$
t = \sum_{i=1}^p v_i \otimes w_i
$
be an independent representation of $t;$ then, the bilinear mapping
$$
\beta : L_t \times R_t \rightarrow \mathbb{K},
\qquad
\beta(v_i, w_j) = \delta_{ij}
$$
depends only on the tensor $t.$
\end{proposition}

\begin{proof}
A preliminary remark. Notice that
\begin{multline*}
\sum_i (v_i \cdot T) \otimes w_i
:=
\sum_i \left(\sum_h v_h T_{hi}\right) \otimes w_i =
\\
=
\sum_i v_i \otimes \left(\sum_k T_{ik} w_k\right)
:=
\sum_i v_i \otimes (T \cdot w_i).
\end{multline*}
For every $T$ non singular, we have
$$
t = \sum_i (v_i \cdot T) \otimes (T^{-1} w_i) = \sum_i v_i \otimes w_i.
$$
In plain words, for any two minimal representations
$$
t = \sum_i v_i \otimes w_i = \sum_i v'_i \otimes w'_i,
$$
there is a unique non singular matrix $T$ such that
$
v'_i = v_i \cdot T,
$
$
w'_i = T^{-1} \cdot w_i.
$

\noindent
Let $\beta$ be the bilinear form associated to the representation
$t = \sum_i v_i \otimes w_i,$ that is, such that
$\beta(v_i, w_j) = \delta_{ij}.$

\noindent
Now, it is clear that
$$
\beta\left(v_i \cdot T, w_j \right) = \beta\left(v_i, T \cdot w_j
\right),
$$
for every $T.$ This immediately implies
$$
\beta\left(v_i \cdot T, T^{-1} \cdot w_j \right)
=
\beta\left(v_i, w_j \right),
$$
that is, the bilinear form $\beta$ is the same as the bilinear form
associated to the representation $t = \sum_i v'_i \otimes w'_i.$
\end{proof}
\end{enumerate}


\begin{thebibliography}{99}
\def\topsep{0pt}
\def\parsep{0pt plus 5pt minus 1pt}
\def\itemsep{-0.5ex} 
\small               


\bibitem{Abe}
Abe, E.: {\it Hopf algebras},
Cambridge Tracts in Mathematics, 74.
Cambridge University Press, Cambridge-New York, (1980)
%
\bibitem{BBR}
Barnabei, M., Brini, A., Rota, G.-C.:  {\it On the Exterior Calculus
of Invariant Theory}.  J. of Algebra {\bf 96}, 120-160 (1985)
%
\bibitem{berget}
Berget, A.: {\it Tableaux in the Whitney module of a matroid}.
S´eminaire Lotharingien de Combinatoire {\bf 63}, Article B63f, pp. 17  (2010)
%
\bibitem{Bravi}
Bravi, P., Brini, A.:  {\it Remarks on Invariant Geometric
Calculus, Cayley-Grassmann Algebras and Geometric Clifford
Algebras}. In:  Crapo,H., Senato, D. (eds.) Algebraic Combinatorics and Computer Science,
pp. 129--150, Springer, Milano (2001)
%
\bibitem{brini}
Brini, A.: {\it Combinatorics, Superalgebras, Invariant Theory and
Representation Theory}. S´eminaire Lotharingien de
Combinatoire {\bf 55},  Article B55g, pp. 117 (2007)
%
\bibitem{pri}
Brini, A.: {\it Private communication to A. Berget}. (2009)
%
\bibitem{BPT}
Brini A., Palareti A., Teolis  A.: {\it Gordan--Capelli series in
superalgebras}.  Proc. Natl. Acad. Sci. USA {\bf 85}, 1330--1333 (1988)
%
\bibitem{BrT}
Brini, A., Regonati,F., Teolis,A.:  {\it Grassmann geometric calculus, invariant theory and superalgebras}.
In:  Crapo,H., Senato, D. (eds.) Algebraic Combinatorics and Computer Science,
pp. 151--196, Springer, Milano (2001)
%
\bibitem{bt}
Brini A., Teolis  A.: {\it Young-Capelli symmetrizers in
superalgebras}.  Proc. Natl. Acad. Sci. USA {\bf 86}, 775--778 (1989)
%
\bibitem{abl}
Brini, A., Regonati, F.,  Teolis, A.: {\it The Method of Virtual
Variables and Representations of Lie Superalgebras}. In: Ablamowicz, R. (ed.)
Clifford Algebras - Applications to Mathematics, Physics, and
Engineering, Progress in Mathematical Physics, vol. 34, pp. 245--263, Birkhauser, Boston (2004)
%
\bibitem{Cliff}
Clifford, W.K.:  {\it Application of Grassmann's Extensive Algebra}.
Amer. J. of Mathematics, 350--358  (1878)
%
\bibitem{crapo}
Crapo, H.:  {\it An Algebra of Pieces of Space -- Hermann Grassmann to
Gian Carlo Rota}. In: Damiani, E., D'Antona, O., Marra, V. and Palombi, F.(eds.)
From Combinatorics to Philosophy, pp. 61--90, Springer,  (2009)
%
\bibitem{crasch}
Crapo, H., Schmitt,W.:  {\it The Whitney algebra of a matroid}.
J. Comb. Theory {\bf A} {\bf 91}, 215--263 (2000)
%
\bibitem{dieu}
Dieudonne, J.:  {\it The tragedy of Grassmann}.  Linear and Multilinear Algebra {\bf 8}, 1-14  (1979/80)
%
\bibitem{Fau}
Fauser, B.: {\it A treatise on quantum Clifford algebra}, Kostanz, Abilitationsschrift,
arXiv:math.QA/0202059 (2002)
%
\bibitem{Ford}
Forder, H.D.: {\it The Calculus of Extension}, Chelsea Publishing Co., New York (1960)
%
\bibitem{gra1}
Grassmann, H.G.: {\it Die Lineale
Ausdehnungslehre}, Verlag von Otto Wigand, Leipzig (1844). Translated by Kannenberg, L.C.: {\it The
Ausdehnungslehre of 1844 and Other Works},  La Salle: Open
Court Publ., Chicago (1995)
%
\bibitem{GRS}
Grosshans, F.D.,  Rota, G.-C., Stein, J.A.: {\it Invariant
theory and Superalgebras,}, Amer. Math. Soc., Providence, RI (1987)
%
\bibitem{Hawr}
Hawrylycz, M.:
{\it Arguesian identities in invariant theory}.
Advances in Math. {\bf 122}, 1–48 (1996)
%
\bibitem{Hest}
Hestenes, D., Sobczyk, G., {\it Clifford Algebra to Geometric
Calculus}, Reidel (1984)
%
\bibitem{li}
Li, H., {\it Invariant Algebras and Geometric Reasoning}, World
Publishing Co. (2008)
%
\bibitem{Main}
Mainetti, M., Yan, C.H.: {\it Arguesian identities in linear lattices}.
Advances in  Math. {\bf 144}, 50–93 (1999)
%
\bibitem{Mou}
Mourrain, B.:
{\it Approche effective de la théorie des invariants des groupes classiques},
PhD Thesis, Ecole Polytechnique  (1991)
%
\bibitem{Moust}
Mourrain, B., Stolfi, N.: {\it Computational symbolic geometry}. In: White, N. (ed.),
Invariant Methods in Discrete and Computational Geometry,  pp.107--139,
Reidel, (1995)
%
\bibitem{Peano}
Peano, G.: {\it Calcolo geometrico secondo l'Ausdehnungslehre di
H. Grassmann}, Fratelli Bocca Editori, Torino (1888)
%
\bibitem{pet}
Petsche, H.J. (Chairperson): {\it Grassmann Bicentennial
Conference }, Potsdam and Sczcecin (2009)
%
\bibitem{reg}
Regonati, F.: {\it On the Combinatorics of Young-Capelli Symmetrizers}.
S´eminaire  Lotharingien de Combinatoire {\bf 62}, Article B62d, pp.36 (2010)
%
\bibitem{Schu}
Schubring, G. (Chairperson):  {\it The Grassmann Jubilaeum.
International Conference on the occasion of the 150th anniversary
of the publication of the Ausdehnungslehre }, Rugen and Sczcecin  (1994)
%
\bibitem{Schu1}
G Schubring (ed.): {\it Hermann  Grassmann (1809-1877) : visionary mathematician,
scientist and neohumanist scholar}, Kluwer, Dordrecht (1996)

%
\bibitem{Stew}
Stewart, I.: {\it Hermann Grassmann was right}.
Nature {\bf 321}, 17  (1986)
%
\bibitem{Swe}
Sweedler, Moss E.: {\it Hopf algebras},
Mathematics Lecture Note Series, W. A. Benjamin, Inc., New York (1969)
%
\bibitem{Weyl}
Weyl, H.: {\it The Classical Groups. Their invariants and representations},
Princeton University Press, Princeton , New York (1946)
%
\bibitem{whi1}
White, N.L.: {\it The bracket ring of a combinatorial geometry I}.
Trans. of Amer. Math Soc. {\bf 202}, 79--95  (1975)
%
\bibitem{whitney}
Whitney, H.: {\it On the Abstract Properties of Linear Dependence}.
Amer. J. Math. {\bf 57}, no. 3, 509--533 (1935)
\end{thebibliography}
\end{document}